\documentclass[11pt,a4paper,oneside]{article}
\usepackage[margin=0.8in]{geometry}
\usepackage{kbordermatrix}
\usepackage{graphicx} 

\usepackage{amsmath}
\usepackage{amssymb}
\usepackage{mathtools}
\usepackage{mathcomp}
\usepackage{textcomp}
\usepackage{amsthm}

\usepackage{bm}
\usepackage{url}
\newtheorem{thm}{Theorem}[section]\theoremstyle{plain}
\newtheorem{theorem}[thm]{Theorem}\theoremstyle{plain}
\newtheorem{proposition}[thm]{Proposition}\theoremstyle{plain}
\newtheorem{lemma}[thm]{Lemma}\theoremstyle{plain}
\theoremstyle{plain}

\theoremstyle{plain}
\theoremstyle{plain}
\newtheorem{corollary}[thm]{Corollary}\theoremstyle{plain}
\theoremstyle{plain}
\newtheorem{conj}{Conjecture}\theoremstyle{plain}
\theoremstyle{plain}
\newtheorem{problem}[conj]{Problem}\theoremstyle{plain}
\theoremstyle{plain}
\newtheorem{conjecture}[conj]{Conjecture}\theoremstyle{plain}
\theoremstyle{plain}

\DeclareMathOperator{\rank}{rank}

\newcommand{\cW}{{\cal W}}

\newcommand{\cL}{{\cal L}}

\newcommand{\s}{{\rm si}}

\DeclareMathOperator{\rd}{rd}

\DeclareMathOperator{\Gram}{Gram}

\title{Super Stable Tensegrities and the Colin de Verdi\`{e}re Number $\nu$}
\author{Ryoshun Oba and Shin-ichi Tanigawa\thanks{Department of Mathematical Informatics, Graduate School of Information Science and Technology, University of Tokyo, 7-3-1 Hongo, Bunkyo-ku, 113-8656,  Tokyo Japan. email: {\tt ryoshun\_oba@mist.i.u-tokyo.ac.jp, tanigawa@mist.i.u-tokyo.ac.jp}}}
\begin{document}
\maketitle

\begin{abstract}

    A super stable tensegrity introduced by Connelly in 1982 is a globally rigid discrete structure made from stiff bars and struts connected by cables with tension.
    We introduce the super stability number of a multigraph as the maximum dimension that 
    a multigraph can be realized as a super stable tensegrity, and show that it equals the Colin de Verdi\`{e}re number~$\nu$ minus one.
    As a corollary we obtain a combinatorial characterization of multigraphs that can be realized as three-dimensional super stable tensegrities.
    We also show that, for any fixed $d$, there is an infinite family of $3$-regular graphs that can be realized as $d$-dimensional injective super stable tensegrities.
\end{abstract}


\section{Introduction}\label{sec:intro}
A tensegrity is a stable discrete structure made from stiff bars and struts\footnote{A strut is a rod-shaped material which can stretch but cannot shrink.} connected by cables with tension. See Figure~\ref{fig:tensegrity}.
Since its invention by the American sculptor Kenneth Snelson, tensegrities have attracted wide attention not only in structural engineering  but also in  mathematics, see, e.g.,~\cite{CG}.
One of the challenges in the study of tensegrities is to understand the impact of the combinatorics underlying the stable tensegrities.
The aim of this paper is to relate such a combinatorial question in the tensegrity analysis to a topic in spectral graph theory and derive combinatorial characterizations of the graphs of stable tensegrities.

\begin{figure}[h]
    \centering
    \begin{minipage}{0.4\textwidth}
    \centering
    \includegraphics[scale=0.13]{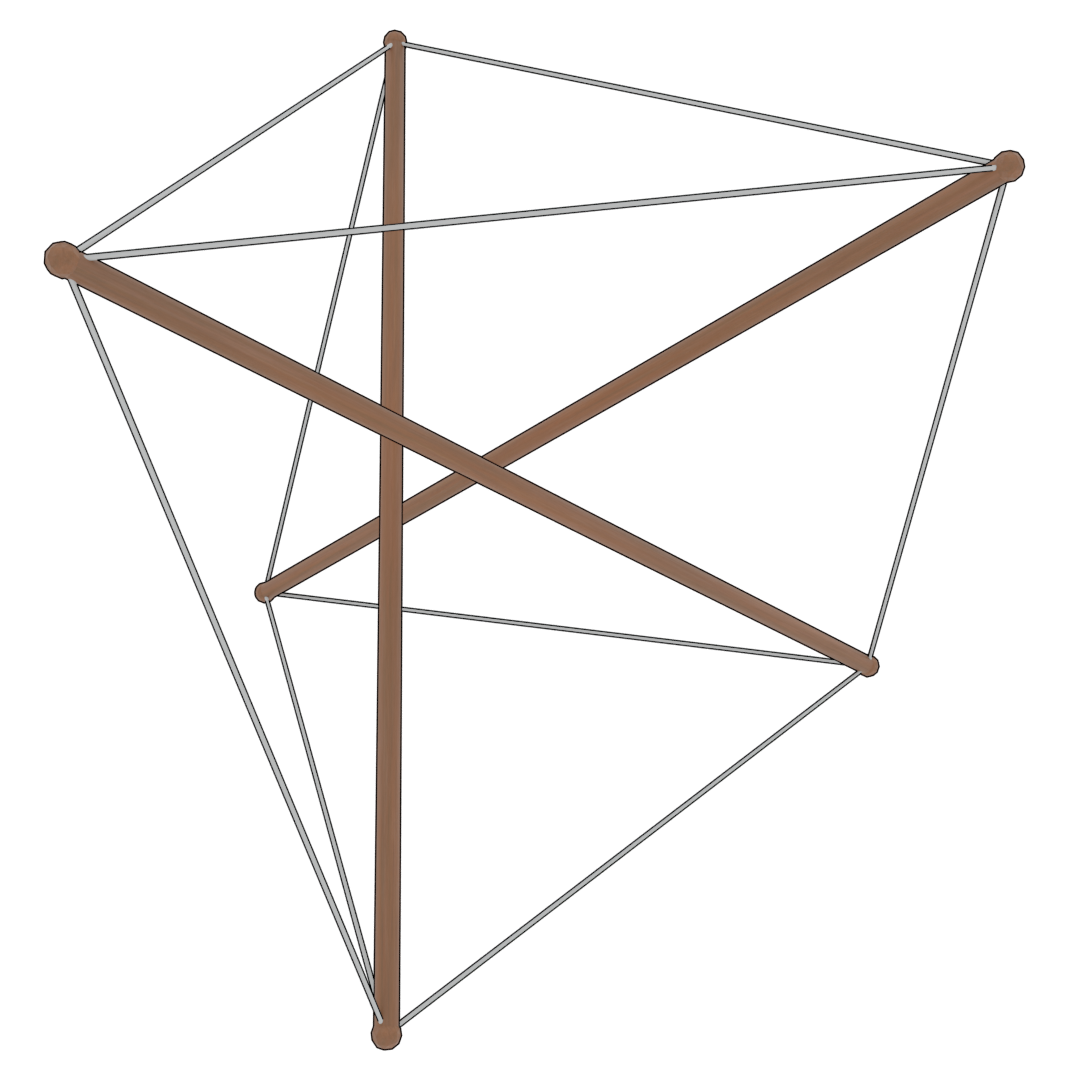}
    \par 
    (a)
    \end{minipage}
    \begin{minipage}{0.4\textwidth}
    \centering
    \includegraphics[scale=0.13]{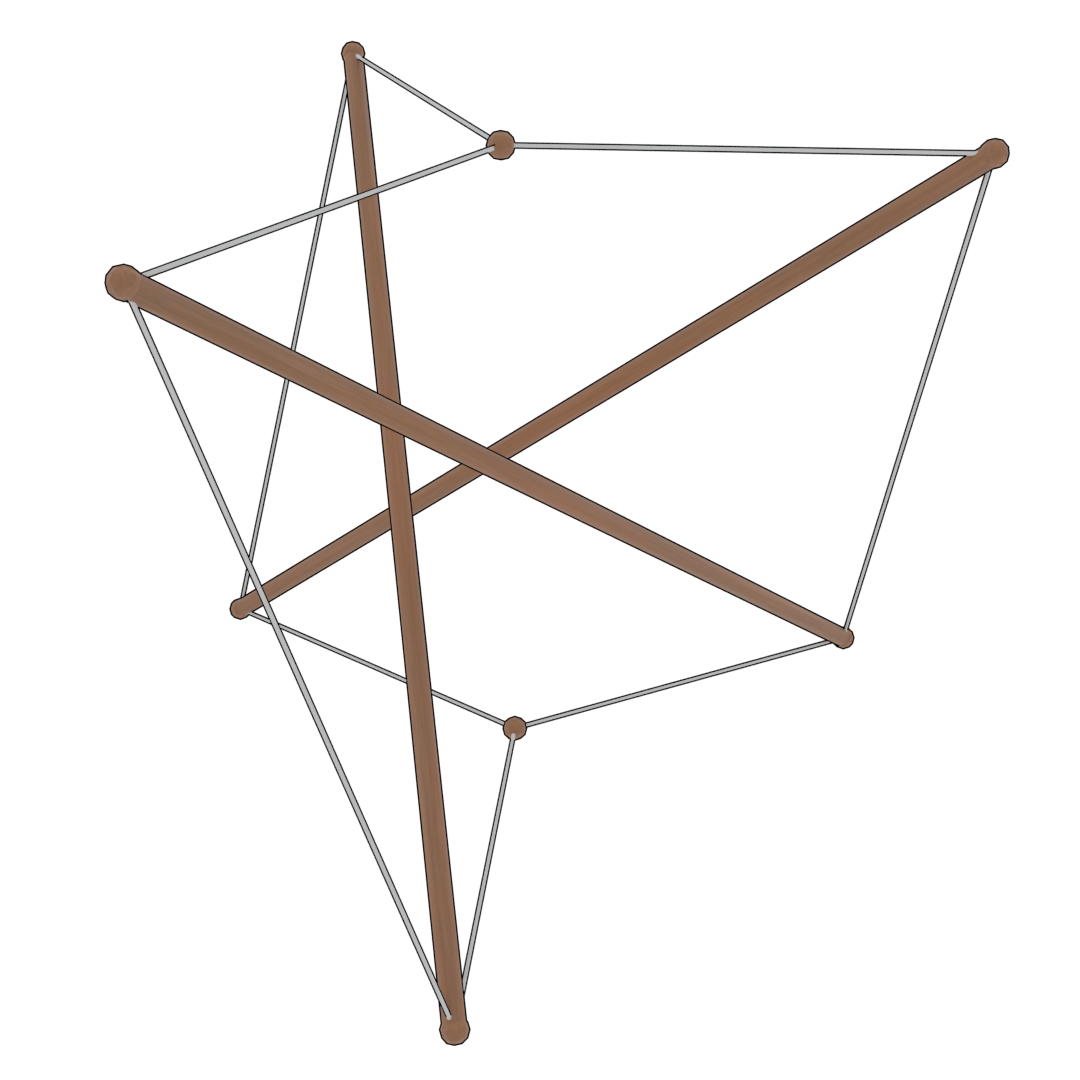}
    \par
    (b)
    \end{minipage}
    \caption{(a) Prism tensegrity and (b) dihedral star tensegrity. They are three-dimensional tensegrity realizations of $K_{2,2,2}$ and $Q_3$, respectively.}
    \label{fig:tensegrity}
\end{figure}

Mathematically, a {\em tensegrity} in $\mathbb{R}^d$ is a triple $(G,\sigma, p)$,
where $G$ is a multigraph\footnote{Throughout the paper, a multigraph is an undirected graph which may contain parallel edges but no loops.} with a finite vertex set  $V(G)$ and edge set $E(G)$, $\sigma:E(G)\rightarrow \{+, -\}$ is a sign function, and $p:V(G)\rightarrow \mathbb{R}^d$ is  a {\em point configuration}\footnote{In this paper, point configurations may not necessarily be injective.} in the $d$-dimensional Euclidean space  $\mathbb{R}^d$.
A tensegrity $(G,\sigma, p)$ is called a {\em tensegrity realization} of a multigraph $G$.
A tensegrity $(G,\sigma, p)$ models a physical structure by regarding 
each edge $e$ with $\sigma(e)=+$ as a {\em cable} and each edge $e$ with $\sigma(e)=-$ as a {\em strut}.
Thus,  a  {\em deformation} of a tensegrity $(G,\sigma, p)$ is defined as another tensegrity $(G,\sigma,q)$ with the same underlying signed graph $(G,\sigma)$ and 
\[
    \begin{array}{cc}
        \|p(i)-p(j)\| \geq \|q(i)-q(j)\| & \text{if $\sigma(ij)=+$} \\
        \|p(i)-p(j)\| \leq \|q(i)-q(j)\| & \text{if $\sigma(ij)=-$},
    \end{array}
\]
for every edge $ij\in E(G)$,
where $\|\cdot\|$ denotes the Euclidean norm.
A deformation represents a possible deformed tensegrity of the given tensegrity $(G,\sigma,p)$  under the cable and strut constraints of $(G,\sigma,p)$. 

It should be remarked that the notation defined above is slightly different from the standard, e.g.,~\cite{C82,CG,RW}.
Conventionally, the underlying graph of a tensegrity is always a simple graph and a sign function $\sigma$ allows to take value $0$ to represent a bar (which cannot stretch or shrink). 
In our model, a {\em bar} can be represented by parallel edges with opposite signs, and hence ours is essentially identical to the standard model.
However, our unconventional  notation using multigraphs turns out to be crucial in the development of a new combinatorial theory.

Once we have defined deformations, the rigidity of tensegrities can be naturally defined as follows.
Two point configurations $p, q$ of the same set in $\mathbb{R}^d$ are said to be {\em congruent} if one is the image of a Euclidean isometry applied to the other,
and two tensegrities $(G, \sigma, p)$ and $(G, \sigma, q)$ are {\em congruent} if 
$p$ and $q$ are congruent.
A tensegrity $(G,\sigma, p)$ in $\mathbb{R}^d$ is {\em globally rigid} if 
any deformation  of $(G,\sigma,p)$ in $\mathbb{R}^d$ is congruent to $(G,\sigma,p)$.
It is called a {\em locally rigid} if there is an open neighborhood $N$ of $p$ in the space of point  configurations  such that 
any deformation $(G,\sigma,q)$ of $(G,\sigma,p)$ with $q\in N$ is congruent to $(G,\sigma,p)$.

Though testing the local/global rigidity of bar-joint frameworks is known to be a hard problem, the problem becomes tractable if point configurations are assumed to be generic, see e.g.~\cite{JW}. However, for tensegrities, the generic assumption does not seem as powerful as the ordinary bar-joint case, see, e.g.~\cite{G21,JJK}.
Hence, the standard approach  in practice is to check a sufficient condition. Commonly used sufficient conditions for rigidity are infinitesimal rigidity~\cite{RW}, prestress stability~\cite{CW}, and super stability~\cite{C82}.

In this paper, we focus on super stability. 
This concept was introduced by Connelly~\cite{C82} in 1982 and it is currently one of few general sufficient conditions for the global rigidity of tensegrities.
Roughly speaking, a tensegrity becomes super stable if the point configuration is the unique minimizer of a convex harmonic potential function, and hence  it has an advantage of being robust with respect to perturbations. Such an advantage has been well appreciated and used in the design of new tensegrities in structural engineering, see, e.g.,~\cite{CG,ZO}. However, unlike the recent rapid development of the global rigidity theory of bar-joint frameworks (see, e.g.,~\cite{JW}), little mathematical progress has been made for the combinatorics of tensegrities.

A tensegrity is said to be {\em $d$-dimensional}\footnote{Throughout the paper, 
we distinguish a {\em tensegrity in $\mathbb{R}^d$} and a {\em $d$-dimensional tensegrity}.
A tensegrity in $\mathbb{R}^d$ just means a tensegrity each of whose point is in $\mathbb{R}^d$ and the affine span of the points may not be $d$-dimensional.}
if the dimension of the affine span of its point configuration is equal to $d$. 
In this paper, we study the class of multigraphs which can be realized as $d$-dimensional super stable tensegrities for a given positive integer $d$.
We show that (the complement of) this multigraph class admits a characterization in terms of forbidden minors,
and give a complete list of minors in the case when $d\leq 3$.
Specifically, we obtain the following.
\begin{theorem}\label{thm:characterization_super_stable}
Let $G$ be a multigraph.
Then $G$ has a three-dimensional super stable tensegrity realization if and only if it contains a multigraph in  Figure~\ref{fig:list} as a minor.
\end{theorem}
The proof even shows that, if $G$ has a three-dimensional super stable realization, then such a realization of $G$ can be obtained from a super stable realization of a graph in Figure~\ref{fig:list} by a sequence of edge addition and vertex splitting (that is, the inverse of edge contraction) operations. 
In this sense, super stable realizations of the graphs in Figure~\ref{fig:list} are the bases of the construction. 
Moreover, any graph in Figure~\ref{fig:list} except $K_5$ can be constructed from the cube graph $Q_3$ (the second left graph in the top row of Figure~\ref{fig:list}) by a sequence of Y-$\Delta$ operations. Since a Y-$\Delta$ operation preserves super stability, we may even consider that $K_5$ and $Q_3$ are the bases of the underlying multigraphs of the three-dimensional super stable tensegrities.
Interestingly, a super stable realization of $Q_3$ has been well investigated as a {\em dihedral star-shaped tensegrity} in the engineering context~\cite{ZGCO}, see Figure~\ref{fig:tensegrity}(b).

\begin{figure}[t]
\centering
\begin{minipage}{0.24\textwidth}
\centering
\includegraphics[scale=1]{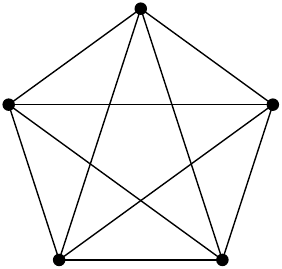}
\par
$K_5$
\end{minipage}
\begin{minipage}{0.24\textwidth}
\centering
\includegraphics[scale=1]{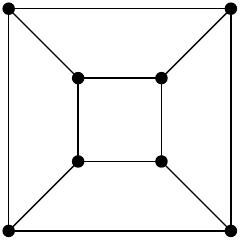}
\par
$Q_3$
\end{minipage}
\begin{minipage}{0.24\textwidth}
\centering
\includegraphics[scale=1]{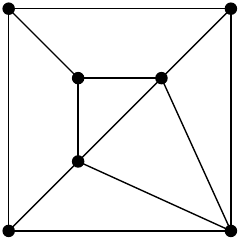}
\par
\ 
\end{minipage}
\begin{minipage}{0.24\textwidth}
\centering
\includegraphics[scale=1]{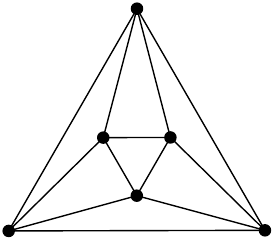}
\par
$K_{2,2,2}$
\end{minipage}

\medskip

\begin{minipage}{0.24\textwidth}
\centering
\includegraphics[scale=1]{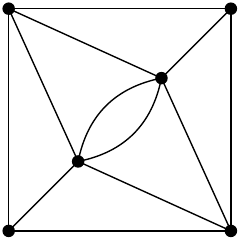}
\par
\ 
\end{minipage}
\begin{minipage}{0.24\textwidth}
\centering
\includegraphics[scale=1]{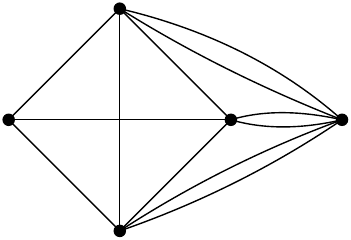}
\par
\end{minipage}
\begin{minipage}{0.24\textwidth}
\centering
\includegraphics[scale=1]{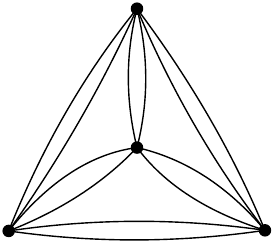}
\par
$K_4^=$
\end{minipage}
\caption{The list of minors in the characterization of Theorem~\ref{thm:characterization_super_stable}.}
 \label{fig:list}
 \end{figure}

The proof of our characterization is actually a corollary of van der Holst's theorem~\cite{H02} on the characterization of multigraphs with bounded Colin de Verdi\`{e}re number $\nu$.
We will establish an exact relation between the super stability of tensegrities and the Colin de Verdi\`{e}re number $\nu$, a spectral-based graph parameter introduced by Colin de Verdi\`{e}re~\cite{C98} in 1998.

We should also remark that the realizations in Theorem~\ref{thm:characterization_super_stable} are not guaranteed to be injective, and some cables or struts may  have zero length in the realizations. However, the technique for constructing super stable realizations with respect to minor operations can be further extended, and we are able to derive sufficient combinatorial conditions for injective realizability. Below we give an example of our sufficient conditions.
\begin{theorem}\label{thm:Kt}
Suppose a 2-connected multigraph $G$ contains $K_{d+2}$ as a minor.
Then $G$ has a $d$-dimensional injective super stable tensegrity realization.
\end{theorem}

The dihedral star tensegrity in Figure~\ref{fig:tensegrity}(b) has attracted attention since the underlying graph is $3$-regular and is much sparser than tensegrities satisfying Maxwell count condition~\cite{CG}.
Since there is an infinite family of 3-regular graphs that contains $K_{d+2}$ as a minor, 
Theorem~\ref{thm:Kt} implies that
there is an  infinite number of such examples in any dimension.
\begin{corollary}
For any fixed $d$, there is an infinite family of 3-regular graphs that can be realized as a $d$-dimensional injective super stable tensegrity. 
\end{corollary}

As a dual concept of super stable realizability, we also introduce the {\em realizable dimension of multigraphs}, 
which is a tensegrity extension of the realizable dimension of graphs given by Belk and Connelly~\cite{BC}.
We give a forbidden-minor  characterization of multigraphs with  realizable dimension at most one.

\medskip

The paper is organized as follows.
In Section~\ref{sec:super_stable} we review Connelly's super stability theorem and give the definition of  super stable tensegrities. We then introduce new multigraph parameters, super stability number $\lambda$ and realizable dimension ${\rm rd}$, which are the main objects of this study.
In Section~\ref{sec:construction} we show how to build super stable tensegrities based on simple graph operations. 
As a corollary, we prove the minor monotonicity of the super stability number $\lambda$. We will also discuss how to construct injective tensegrity realizations.
A relation between super stability and the Colin de Verdi\`{e}re number $\nu$ is discussed in Section~\ref{sec:characterization_lambda}.
In Section~\ref{sec:relation} we give relations among $\lambda$, ${\rm rd}$, and other basic graph parameters. 
In Section~\ref{sec:characterization_rd} we discuss  the characterization problem of multigraphs with bounded realizable dimension.

\vspace{1em}

Throughout the paper we use the following basic terminology from graph theory. 
For a multigraph $G$, $V(G)$ and $E(G)$ denote the vertex set and the edge set of $G$, respectively. 
$N_G(v)$ denotes the set of vertices adjacent to $v$ in $G$
and $E_G(v)$ denotes the set of edges incident to $v$ in $G$.
We sometimes denote an edge $e$ by $e=uv$ when the endvertices of an edge $e$ are $u$ and $v$.

The {\em simplified graph} $\s(G)$ of a multigraph $G$ is a simple graph obtained from $G$ by replacing each parallel edges with a single edge.
For a simple graph $H$, let $H^=$ be the multigraph obtained from $H$ by doubling each edge.
For a multigraph $G$, let $G^==(\s(G))^=$.

In this paper a {\em bar-joint framework} is defined as a tensegrity $(G^=, \sigma, p)$ such that each parallel edge classes has both positive and negative edges (so that the interpoint distance between $p(u)$ and $p(v)$ is fixed for each edge $uv$ in $G$).
We sometimes abbreviate it to $(G,p)$.

\section{Super Stability Number and Realizable Dimension}\label{sec:super_stable}
In this section, we give a formal definition of super stability of tensegrities,
and then introduce two new multigraph parameters, the super stability number and the realizable dimension.

\subsection{Super Stability of Tensegrities}
We follow a modern exposition of super stability in terms of semidefinite programming, see, e.g.,\cite{GT}.

Let $V=\{1,2,\dots, n\}$ be a finite set.
We identify a point configuration $p:V \rightarrow \mathbb{R}^d$ with a matrix $P \in \mathbb{R}^{d \times n}$ whose $i$th column vector is $p(i)$ for all $i \in V$.
For a configuration $p:V \rightarrow \mathbb{R}^d$, the Gram matrix of $p$, denoted by $\Gram(p)$, is a $n \times n$ matrix whose $ij$-th entry is $p(i)^\top p(j)$. Equivalently, $\Gram(p)=P^\top P$.
Since the properties of tensegrities we are interested in are invariant by translation, 
we may always assume that the center of gravity of $\{p(i):i \in V\}$ is at the origin,
that is, $P\mathbf{1}_n=\bm{0}$, where $\mathbf{1}_n$ denotes the all-one vector in $\mathbb{R}^n$.

Given a multigraph $G=(V,E)$ with an edge weight $\omega:E\rightarrow \mathbb{R}$, its weighted {\em Laplacian} $L_{G,\omega}$ is defined by 
\[
L_{G,\omega}:=\sum_{e=ij\in E} \omega(ij) F_{ij},
\]
where
\[
F_{ij}:=(\bm{e}_i-\bm{e}_j)(\bm{e}_i-\bm{e}_j)^\top
\]
and $\bm{e}_i$ is the $n$-dimensional vector whose $i$-th coordinate is one and the remaining entries are zeros.
Let $\mathcal{L}^n$ be the set of all Laplacian matrices over all multigraphs with $n$ vertices. Then $\mathcal{L}^n$ is a linear subspace of the space $\mathcal{S}^n$ of symmetric matrices of size $n$ given by
\[
\mathcal{L}^n={\rm span}\{F_{ij}: i,j\in V, i\neq j\}.
\]
The set of positive semidefinite Laplacian matrices is denoted by $\mathcal{L}_+^n$.

Given a $d$-dimensional tensegrity $(G,\sigma,p)$, consider the following semidefinite programming problem $\text{P}_{(G,\sigma,p)}$:
\[
    \begin{array}{lllr}
            \text{($\text{P}_{(G,\sigma,p)}$)} &  \text{maximize}  & 0 \\
             & \text{subject to }  & \langle X, F_{ij} \rangle \leq \|p(i)-p(j)\|^2 & \quad (ij\in E \text{ with } \sigma(ij)=+1) \\
             &              & \langle X, F_{ij} \rangle \geq \|p(i)-p(j)\|^2 & \quad (ij\in E \text{ with } \sigma(ij)=-1) \\
             &              & X \succeq 0\\
             & & X\in \mathcal{L}^n,
        \end{array}
\]
where $\langle\cdot, \cdot\rangle$ denotes the Frobenius inner product and $X \succeq 0$ means that $X$ is positive semidefinite.
The feasible region of $\text{P}_{(G,\sigma,p)}$ describes deformations of $(G,\sigma,p)$ 
in the sense that $(G,\sigma,q)$ is a deformation of $(G,\sigma,p)$ if and only if $X=\Gram(q)$ is feasible in $\text{P}_{(G,\sigma,p)}$.
Hence, if $\text{P}_{(G,\sigma,p)}$ has a unique solution (which must be $X=\Gram(p)$), then $(G,\sigma,p)$ is globally rigid in $\mathbb{R}^d$.
A tensegrity $(G,\sigma,p)$ having the property that $\text{P}_{(G,\sigma,p)}$ has the unique solution is called \emph{universally rigid},
and the property has been extensively studied, see, e.g.,\cite{A11,GT}.


Connelly's super stability theorem given below guarantees the universal rigidity in terms of the existence of a certain equilibrium stress.
An edge weight $\omega:E \rightarrow \mathbb{R}$ is called an \emph{equilibrium stress} of a tensegrity $(G,\sigma,p)$ if 
\begin{align}\label{eq:equilibrium}
     \sum_{j \in N_G(i)} \omega(ij) \left( p(j)-p(i) \right) = 0 \qquad (i\in V)
     \end{align}
holds,
and it is called \emph{proper} if $\sigma(ij)\omega(ij) \geq 0$ for each $ij \in E(G)$. 
It is \emph{strictly proper} if $\sigma(ij)\omega(ij) > 0$ for each $ij \in E(G)$.

\begin{theorem}[Connelly's super stability theorem~\cite{C82}] \label{thm:connelly}
    A $d$-dimensional  tensegrity $(G,\sigma,p)$ is universally rigid (and hence globally rigid) if the following two conditions are satisfied.
    \begin{description}
        \item[(i) Stress Condition:] There exists a strictly proper equilibrium stress $\omega:E(G) \rightarrow \mathbb{R}$ such that $L_{G,\omega} \succeq 0$ and $\dim \ker L_{G,\omega}=d+1$.
        \item[(ii) Conic Condition:] The edge directions of $(G,p)$ do not lie on a conic at infinity, i.e., there exists no nonzero matrix $S \in \mathcal{S}^d$ such that
        \[
            \left( p(j)-p(i) \right)^\top S \left(p(j)-p(i)\right) = 0\quad \forall ij \in E(G).
        \]
    \end{description}
\end{theorem}

A $d$-dimensional  tensegrity $(G,\sigma,p)$  is said to be {\em super stable} if it satisfies the stress condition and the conic condition  in Theorem~\ref{thm:connelly}.
For a tensegrity $(G,\sigma,p)$, the matrix $L_{G,\omega}$ defined by a (strictly proper) equilibrium stress $\omega$ is called an \emph{equilibrium stress matrix} (see~\cite{C82}).

\subsection{Super Stability Number}
Based on the concept of super stability, we now introduce a new graph parameter, the {\em super stability number}, as follows.
    For a connected multigraph $G$, the super stability number of $G$ is defined by
    \[
    \lambda(G):=\max\{d\in \mathbb{Z}: \text{$G$ has a $d$-dimensional super stable tensegrity realization}\}.
    \]
    For a multigraph $G$ in general, the super stability number is defined by
    \[
      \lambda(G):=\max\{\lambda(H): \text{$H$ is a connected component of $G$}\}.
\]

\medskip
\noindent
{\bf Example  1 (Tree).} 
Suppose $G$ is a single-edge graph, that is, $V(G)=\{u,v\}$ and $E(G)=\{uv\}$.
If a tensegrity realization $(G,\sigma,p)$ of $G$ is one-dimensional,
then there is no nonzero edge weight $\omega:E(G)\rightarrow \mathbb{R}$ satisfying the equilibrium condition (\ref{eq:equilibrium}).
Hence, $\lambda(G)<1$.
On the other hand, if $(G,\sigma,p)$ is a zero-dimensional tensegrity realization such that $\sigma(uv)=+$, 
then the edge weight $\omega$ with $\omega(uv)=1$ satisfies the equilibrium condition (\ref{eq:equilibrium}).
Then, $L_{G,\omega}=\begin{pmatrix} 1 & -1\\ -1 & 1\end{pmatrix}$, which is 
a rank-one positive semidefinite matrix. 
Thus,  $(G,\sigma,p)$ satisfies the stress condition for super stability.
In this case, the conic condition is null, and hence $(G,\sigma,p)$ is super stable,
and $\lambda(G)=0$.

This argument can be adapted to a general tree to show  that $\lambda(G)=0$ if $G$ is a tree (or forest).

\medskip
\noindent
{\bf Example 2 (Cycle).} Suppose $G$ is a cycle $C_n$ with $V(C_n)=\{v_1,\dots, v_n\}$
and $E(C_n)=\{v_iv_{i+1}:i=1,\ldots,n\}$, where $v_{n+1}=v_1$.
If a tensegrity realization $(C_n,\sigma,p)$ of $G$ is two-dimensional,
then there is a vertex $v_i$ such that 
$p(v_{i-1})-p(v_i)$ and $p(v_{i+1})-p(v_i)$ are not colinear,
or $p(v_{i-1})=p(v_i)$ and $p(v_{i+1})\neq p(v_i)$
(or $p(v_{i-1})\neq p(v_i)$ and $p(v_{i+1})= p(v_i)$).
In either case, the equilibrium condition cannot hold for any nonzero edge weight $\omega$.
So $\lambda(C_n)<2$ follows.

On the other hand, consider a one-dimensional tensegrity $(C_n,\sigma,p)$ such that $p(v_1), p(v_2), \ldots, p(v_n)$ are placed evenly spaced apart in this ordering on the line and $\sigma(v_1v_n)=-1$ and $\sigma(v_iv_{i+1})=+1$ for $i=1,\dots, n-1$.
Then, the edge weight $\omega$ given by $\omega(v_nv_1)=-1$ and $\omega(v_iv_{i+1})=n-1$ for $i=1,\dots, n-1$ satisfies the equilibrium condition and is strictly proper.
One can check that $L_{C_n,\omega}$ is positive semidefinite with $\dim \ker L_{C_n,\omega}=2$. Thus $\lambda(C_n)=1$ follows.

\medskip
\noindent
{\bf Example 3 (Complete graph).}
Consider $K_4$ drawn in the plane such that the four points form the vertices of the unit square.
If the outer cycles are cables and the two diagonal edges are struts, it admits a strictly proper equilibrium stress $\omega$: let $\omega(e)=1$ for each edge in the outer cycle and 
$\omega(e)=-1$ for each diagonal.
Then, $L_{K_4,\omega}=\begin{pmatrix} 1 & -1 & 1 & -1\end{pmatrix}^{\top}\begin{pmatrix} 1 & -1 & 1 & -1\end{pmatrix}$, which is a rank-one positive semidefinite matrix. 
Since it also satisfies the conic condition, $\lambda(K_4)\geq 2$ follows.

More generally, consider the complete graph $K_n$ with $V(K_n)=[n]$ and $p:[n]\rightarrow \mathbb{R}^{n-2}$ in general position.
Since there are $n$ points in $\mathbb{R}^{n-2}$, 
they are affinely dependent, i.e.,  
$\sum_{v\in [n]} a_v p(v)=0$ and $\sum_{v\in [n]}a_v=0$ hold for some $a_v \in \mathbb{R}$. By general position assumption, $a_v \neq 0$ for all $v \in [n]$.
Define $\omega:E(K_n)\rightarrow \mathbb{R}$ by $\omega(uv)=a_ua_v$
and $\sigma:E(K_n)\rightarrow \{-,+\}$ by $\sigma(uv)={\rm sign}(\omega(uv))$.
Then, $\omega$ is a strictly proper equilibrium stress of $(K_n, \sigma, p)$.
Moreover, $L_{K_n,\omega}=\begin{pmatrix} a_1 & \dots & a_n\end{pmatrix}^{\top}\begin{pmatrix} a_1 & \dots  & a_n \end{pmatrix}$, which is a rank-one positive semidefinite matrix. 
Since $(K_n, \sigma, p)$ also satisfies the conic condition, 
we have $\lambda(K_n)\geq n-2$.

We indeed have $\lambda(K_n)=n-2$ since no $(n-1)$-dimensional tensegrity realization of $K_n$ has a strictly proper (or even non-zero) equilibrium stress.

\medskip
\noindent
{\bf Example 4 (Complete multigraphs).} 
Recall that $K_n^=$ is the graph obtained from $K_n$ by duplicating each edge in parallel.
The resulting graph is called the complete multigraph with $n$ vertices.
Consider a tensegrity realization $(K_n^=,\sigma, p)$ of $K_n^=$ such $p:[n]\rightarrow \mathbb{R}^{n-1}$ is in general position and $\sigma(e_1)=-\sigma(e_2)$ for each parallel pair of edges  $e_1, e_2$. 
Then $(K_n^=,\sigma, p)$ satisfies the conic condition.
To see the stress condition,  consider an edge weight $\omega$ by setting $\omega(e_1)=-\omega(e_2)\neq 0$ for each parallel pair of edges  $e_1, e_2$.
Then, $\omega$ is a strictly proper equilibrium stress (since the net stress between two points is zero) and the resulting $L_{K_n^=,\omega}$ is the zero matrix (which has nullity $n$).
Hence, $(K_n^=,\sigma, p)$ is super stable and  $\lambda(K_n^=)=n-1$ follows.

\medskip

The discussion in Example 4 is worth emphasizing. 
Even in a strictly proper stress of a tensegrity realization,
the net stress between two vertices $u$ and $v$ can be zero if there are parallel edges between $u$ and $v$.
This is consistent with the conventional bar-joint case, where the value of a stress can be zero along a bar in Connelly's super stability theorem.
(Recall that parallel edges represent a bar constraint in our tensegrity model.)

\subsection{Spectral view of super stability}
We now explain how the super stability number can be understood from the view of spectral graph theory. The key is an analogy between the conic condition and the Strong Arnold Property from the Colin de Verdi\`{e}re number. Such an analogy (or a connection) has been already observed in e.g.,\cite{CGT,JT,LV} and we will exploit it in this paper.

For a symmetric $n\times n$ matrix $A$, a {\em kernel matrix} $P$ of $A$ is a matrix whose row vectors form an orthogonal basis of  $\ker A$.
Each kernel matrix $P$ defines a corresponding configuration of $n$ points by regarding each column as a point in a Euclidean space.
The corresponding point configuration is called a {\em kernel representation}.
If $A$ is a weighted Laplacian of $G$ with $\dim \ker A=d+1$, then we can take a kernel matrix $P$ such that the last row of a kernel matrix $P$ is the all-one vector.
Let $P'$ be the matrix obtained from $P$ by ignoring the last all-one row vector. $P'$ defines a point configuration in $\mathbb{R}^d$ by regarding each column as a point, which is called a {\em reduced kernel representation} with respect to $A$.

Reduced kernel representations enable us to reduce the problem of finding a super stable realization of a given graph $G$ to a problem of finding a Laplacian matrix supported by $G$ with a certain spectral property.
To see this, we first define the space of Laplacian matrices {\em supported by} $G$ by
\begin{align*}
{\cal L}^*(G):&=\left\{ L_{G,\omega}\mid  \omega:E(G)\rightarrow \mathbb{R}\setminus \{0\}\right\}.
\end{align*}
Note that 
\begin{align*}
{\cal L}^*(G)=\left\{L\in {\cal L}^n \mid \begin{array}{ll} 
L[i,j]= 0 & \text{if $G$ has no  edge between $i$ and $j$} \\
L[i,j]\neq 0 & \text{if $G$ has a single edge between $i$ and $j$} \\
L[i,j]\in  \mathbb{R} & \text{if $G$ has parallel edges between $i$ and $j$}
\end{array} \right\}.
\end{align*}
The (Euclidean) closure of ${\cal L}^*(G)$ is 
\begin{align*}
{\cal L}(G):=\left\{L\in {\cal L}^n \mid \begin{array}{ll} 
L[i,j]= 0 & \text{if $G$ has no  edge between $i$ and $j$} \\
L[i,j]\in  \mathbb{R} & \text{if $G$ has an  edge between $i$ and $j$}
\end{array} \right\},
\end{align*}
which is a linear subspace of ${\cal L}^n$.
Let ${\cal L}_{k}^n$ be the subset of ${\cal L}^n$ consisting of matrices of rank $k$ and ${\cal L}_{+,k}^n$ be the subset of ${\cal L}_k^n$ consisting of positive semidefinite matrices.

The following proposition is fundamental and has been already used in several places before.
\begin{proposition}\label{prop:condition1}
   Let  $G$ be a multigraph, $L\in {\cal L}^*(G)\cap {\cal L}_{+,n-(d+1)}^n$,
   and $p$ a reduced kernel representation of $L$.
   Then there is an edge weight $\omega:E(G)\rightarrow \mathbb{R}\setminus \{0\}$ such that 
   $L_{G,\omega}=L$.
   Moreover, letting $\sigma(e)={\rm sign}(\omega(e))$ for each edge $e$, $\omega$ is a strictly proper equilibrium stress of $(G,\sigma,p)$  such that $L_{G,\omega}\succeq 0$ and $\dim \ker L_{G,\omega}=d+1$.
\end{proposition}
\begin{proof}
If $G$ has a single edge $e$ between $i$ and $j$, let $\omega(e)=-L[i,j]$.
If $G$ has parallel edges $e_1,\ldots,e_k$ ($k \geq 2$) between $i$ and $j$, then let $\omega(e_1), \ldots, \omega(e_k) \in \mathbb{R}\setminus \{0\}$ be arbitrary nonzero real numbers with $\sum_{l=1}^k \omega(e_l)=-L[i,j]$. 
Then $L_{G,\omega}=L$ holds.
Since $p$ is a reduced kernel representation of $L$, we have $PL_{G,\omega}=PL=0$, which is equivalent to the equilibrium condition (\ref{eq:equilibrium}).
Hence, $\omega$ is an equilibrium stress of $(G,\sigma,p)$.
The remaining properties directly follow from $L\in {\cal L}_{+,n-(d+1)}^n$.
\end{proof}

By Proposition~\ref{prop:condition1}, if there is $L \in {\cal L}^*(G)\cap {\cal L}_{+,n-(d+1)}^n$, we can construct a tensegrity $(G,\sigma,p)$ that satisfies the stress condition for super stability with a stress $\omega$ satisfying $L=L_{G,\omega}$. Thus, finding a $d$-dimensional realization satisfying the stress condition reduces to deciding whether
${\cal L}^*(G)\cap {\cal L}_{+,n-(d+1)}^n\neq \emptyset$.

The next question is whether the conic condition for super stability can also be interpreted in terms of Laplacian matrices. This is indeed possible by exploiting an analogy with the Strong Arnold Property. 
The Strong Arnold Property introduced by Colin de Verdi\`{e}re is a non-degeneracy condition for symmetric matrices (see Section~\ref{sec:characterization_lambda} for the definition), and it admits several equivalent characterizations~\cite{G,H96,L}.
By interpreting such characterizations in the space  of Laplacian matrices, we have the following. (See \cite{JT} for a proof.)

\begin{proposition}\label{prop:conic}
    Let $G$ be a graph with $n$ vertices. The following are equivalent for $L\in {\cal L}^*(G)\cap {\cal L}^n_{n-(d+1)}$:
    \begin{description}
    \item[(i)] ${\cal L}(G)$ and ${\cal L}^n_{n-(d+1)}$ intersect transversally\footnote{Two smooth manifolds  $N, N'$ in $\mathbb{R}^m$ {\em intersect transversally at $x\in N\cap N'$} if $T_x N+T_xN'=\mathbb{R}^m$, where $T_x N$ denotes the tangent space of $N$ at $x$.} at $L$ in the space ${\cal L}^n$ of Laplacian matrices.
    \item[(ii)] There is no nonzero $X\in {\cal L}^n$ satisfying $X L=0$ and $\langle X, (\bm{e}_i-\bm{e}_j)(\bm{e}_i-\bm{e}_j)^{\top}\rangle=0$ for every $ij\in E(G)$.
    \item[(iii)] Let $p:V(G)\rightarrow \mathbb{R}^d$ be a reduced kernel representation of $L$.
    Then the edge directions of $(G,p)$ do not lie on a conic at infinity.
    \end{description}
\end{proposition}
The equivalence between Conditions (ii) and (iii) in Proposition~\ref{prop:conic} gives a translation of the conic condition in the language of  Laplacian matrices.
We say that $L$ satisfies \emph{the Euclidean SAP} if it satisfies Condition (ii) in Proposition~\ref{prop:conic}.
    
    Combining Proposition~\ref{prop:condition1} and Proposition~\ref{prop:conic}, we have the following characterization of super stability number. 
    \begin{proposition}\label{prop:algebraic}
        For a connected multigraph $G$,
    \begin{equation}\label{eq:def_L}
        \lambda(G)=\max\left\{\dim \ker L-1: \begin{array}{l} L\in {\cal L}^*(G) \\ L\succeq 0 \\ \text{$L$ satisfies the Euclidean SAP} \end{array}\right\}.
    \end{equation}
    \end{proposition}

\subsection{Realizable dimension}
Belk and Connelly~\cite{BC} introduced the concept of realizable dimension of graphs for the case of bar-joint frameworks.
A natural tensegrity extension can be defined as follows.
A multigraph $G$ is {\em $d$-realizable} if any tensegrity realization of $G$ has a $d'$-dimensional deformation for some $d'\leq d$.
The {\em realizable dimension} ${\rm rd}(G)$ of $G$ is defined by
\[
{\rm rd}(G)=\min\{d\in \mathbb{Z}_{\geq 0}: \text{$G$ is $d$-realizable}\}.
\]

As explained above, a bar-joint framework of a simple graph $G$ can be considered as a tensegrity realization of $G^=$.
Thus,  the $d$-realizability of $G^=$ coincides with the bar-joint case of $d$-realizability due to Belk and Connelly~\cite{BC}.

Observe that, for a tensegrity realization  $(G,\sigma,p)$ of $G$, 
$(G,\sigma,p)$ has a $d'$-dimensional deformation with $d'\leq d$ if and only if 
$\text{P}_{(G,\sigma,p)}$ has a feasible solution $X$ with $\rank X \leq d$.
Hence, if $G$ has a $d$-dimensional universally rigid tensegrity realization  $(G,\sigma,p)$,
then $\text{P}_{(G,\sigma,p)}$ has a unique solution of rank $d$, implying ${\rm rd}(G)\geq d$.
Since any super stable tensegrity is universally rigid by Theorem~\ref{thm:connelly}, we have the following relation between $\rd$ and $\lambda$.
\begin{proposition}\label{prop:rd_lambda}
For any multigraph $G$, $\lambda(G)\leq \rd(G)$.
\end{proposition}

It is not difficult to check the minor monotonicity of realizable dimension.
(Note that a tensegrity may have an edge of length zero.)
\begin{proposition} \label{prop:rd}
    Let $G$ and $H$ be multigraphs. If $H$ is a minor of $G$, then $\rd(H) \leq \rd(G)$.
\end{proposition}
Note also that $\rd(G) \leq |V(G)|-1$ holds as the dimension of the affine space of any point configuration $p$ of $G$ is at most $|V(G)|-1$.

\section{Constructing Super Stable Tensegrities}\label{sec:construction}

We say that a tensegrity $(G,\sigma,p)$ is {\em injective} if $p$ is injective.
In this section we investigate how to  construct injective super stable realizations from those of smaller multigraphs. 
In Subsection~\ref{subsec:removal} we first show that 
the injective realizability as a super stable tensegrity is preserved by edge addition. This is a non-trivial fact because a super stable realization requires a {\em strictly} proper stress according to our definition.
In Subsection~\ref{subsec:contraction} we study the corresponding question for vertex splitting.
This is the technically most difficult part in this paper.
The strategy follows an argument in \cite{JT}, which exploited a technique from \cite{H96} for the minor monotonicity of the Colin de Verdi\`{e}re number. 
In Subsection~\ref{subsec:minor}, we give a list of  corollaries including the minor monotonicity of $\lambda$.

Before moving to the main content, we collect some auxiliary propositions.
The first proposition is useful when checking the conic condition.
\begin{proposition}\label{prop:conic_generic}
    Suppose there is a $d$-dimensional framework $(G,p_0)$ satisfying the conic condition.
    Then 
    \[
    \{q\in (\mathbb{R}^d)^n: \text{$(G,q)$ satisfies the conic condition}\} 
    \]
    is an open dense subset of $(\mathbb{R}^d)^n$,
    where $(\mathbb{R}^d)^n$ denotes the set of all maps from $V(G)$ to $\mathbb{R}^d$
\end{proposition}
\begin{proof}
    Recall that $(G,p)$ does not satisfy the conic condition if and only if there is a nonzero $S\in \mathcal{S}^d$ such that 
    $(p(i)-p(j))^{\top}S(p(i)-p(j))=0$ for any edge $ij\in E(G)$. 
    The latter condition asks whether there is a nonzero solution in a linear system in the entries of $S$.
    This holds if and only if every $d\times d$ minor in the matrix representing the linear system is vanishing. 
    Since each $d\times d$ minor is written as a polynomial in the entries of $p$ and at least one $d\times d$ minor is not vanishing (as the conic condition holds for $(G,p)$), the statement follows. 
\end{proof}

For a square matrix $A=\begin{bmatrix} r & s^{\top} \\ s & T\end{bmatrix}$ with $r\neq 0$,
the {\em Schur complement} at the top left corner
is defined by $T-r^{-1}s s^{\top}$.
We denote it by $A/1$.
The following proposition is well-known~\cite{H68}.
\begin{proposition}\label{prop:schur rank}
Let $A = \begin{bmatrix} r & s^{\top} \\ s & T\end{bmatrix}$ be a symmetric matrix with $r\neq0$. 
Then $\dim \ker A=\dim \ker A/1$ holds.
Moreover, if $r >0$, then $A$ is positive semidefinite if and only if $A/1$ is positive semidefinite. 
\end{proposition}
The following proposition also holds. See \cite{JT} for a proof.
    \begin{proposition}\label{prop:schur}
    Let $A=\begin{bmatrix} r & s^{\top} \\ s & T\end{bmatrix}$ be an $n\times n$ symmetric matrix with corank $d$ and $r\neq 0$. 
    Let $p':\{2,\dots, n\}\rightarrow \mathbb{R}^d$ be a kernel representation of $A/1$, and define
    the extension $p:\{1,\dots, n\}\rightarrow \mathbb{R}^d$ of $p'$ by
    \[
    p(1)=-\frac{1}{r}s^{\top}P'^{\top}
    \]
    where $P'$ denotes the kernel matrix  of $A/1$ corresponding to $p'$.
    Then $p$ is a kernel representation of $A$.
    Moreover, if $(A/1){\bf 1}_{n-1}=0$ and $r+s^{\top}{\bf 1}_{n-1}=0$, then $A {\bf 1}_n=0$.
    \end{proposition}

We also need the following simple corollary of the inverse function theorem.
See, e.g.,~\cite{JT} for a proof.
    \begin{proposition}\label{prop:transversal}
    Let $X, Y, Z$ be smooth manifolds in $\mathbb{R}^m$.
    Suppose that $X$ is an embedded submanifold of $Y$ of codimension at least one, and $X$ and $Z$ intersect transversally at $p\in X\cap Z$.
    Then there is a smooth path $\gamma:[0,1]\rightarrow Y\cap Z$ 
    such that $\gamma(0)=p$ and $\gamma(t)\in Y\setminus X$ for $t\in (0,1]$.
    \end{proposition}
    
\subsection{Edge removal} \label{subsec:removal}
\begin{lemma}\label{lem:edge_removal}
Let $G$ be a connected multigraph and $e$ be an edge of $G$.
Suppose $G-e$ has a $d$-dimensional super stable tensegrity realization.
Then $G$ has a $d$-dimensional super stable tensegrity realization such that the sign of $e$ is positive.

Moreover, if the realization of $G-e$ is injective, then the realization of $G$ can be injective.
\end{lemma}
   \begin{proof}
   Let $e=uv$, $(G-e,\sigma',p')$ be a $d$-dimensional super stable tensegrity realization of $G-e$,
   and $L_{G-e}$ be a weighted Laplacian matrix of $G-e$ that certifies the super stability of $(G-e,\sigma',p')$.
   We may suppose that $p'$ is a reduced kernel representation of $L_{G-e}$.

   If $G-e$ contains an edge $e'$ parallel to $e$, then $L_{G-e}$ also serves as a certificate of super stability of $(G,\sigma,p')$ 
   by extending $\sigma'$ to $\sigma$ by  $\sigma(e')=-\sigma'(e)$.
   Hence we we assume that $G-e$ has no edge parallel to $e$.
   
        By Proposition~\ref{prop:conic},  ${\cal L}(G-e)$ and ${\cal L}^n_{n-(d+1)}$ intersect transversally at $L_{G-e}$ in ${\cal L}^n$.
        Note that ${\cal L}(G)$ can be written by  
        \[
        {\cal L}(G):=\left\{ L+\varepsilon (\bm{e}_u-\bm{e}_v)(\bm{e}_u-\bm{e}_v)^{\top}: L\in {\cal L}(G-e), \varepsilon\in \mathbb{R}\right\}.
        \]
        Since $G-e$ has no edge parallel to $e$,
        ${\cal L}(G-e)$ is a linear subspace of ${\cal L}(G)$ of codimension one.
        Hence, by Proposition~\ref{prop:transversal}, there is a positive number $\overline{\varepsilon}$ and a smooth function $\gamma:[-\overline{\varepsilon}, \overline{\varepsilon}] \rightarrow \mathcal{L}^n$ such that 
        $\gamma(0)=L_{G-e}$ and $\gamma(\varepsilon)\in ({\cal L}(G)\setminus {\cal L}(G-e)) \cap {\cal L}_{n-(d+1)}^n$ for $\varepsilon \in [-\overline{\varepsilon}, \overline{\varepsilon}]\setminus \{0\}$.
        Let $L_{\varepsilon}:=\gamma(\varepsilon)$.
       Then 
       by $L_{\varepsilon}\in {\cal L}(G)\setminus {\cal L}(G-e)$ 
       we have $L_{\varepsilon}\in {\cal L}^*(G)\cap {\cal L}_{n-(d+1)}^n$ and $L_{\varepsilon}\succeq 0$ if $\varepsilon$ is a sufficiently small nonzero number.
       
       Since the entries of $L_{\varepsilon}$ continuously change with respect to $\varepsilon$, we can take a smooth function 
        $[-\overline{\varepsilon}, \overline{\varepsilon}]\ni \varepsilon \mapsto p_{\varepsilon}\in (\mathbb{R}^d)^n$ such that 
        $p_0=p'$ and $p_{\varepsilon}$ is a reduced kernel representation with respect to $L_{\varepsilon}$.
        Then $L_{\varepsilon}$ certifies that $(G,\sigma, p_{\varepsilon})$ satisfies the stress condition for super stability (after extending $\sigma'$ to $\sigma$ such that  $\sigma(e)={\rm sign}\  L_{\varepsilon}[u,v]$).
        Moreover, if $\varepsilon$ is a small positive number, $(G,\sigma, p_{\varepsilon})$ also satisfy the conic condition 
        by Proposition~\ref{prop:conic_generic} and $\sigma(e)>0$. 
        Thus,  $(G,\sigma, p_{\varepsilon})$ is a super stable tensegrity with $\sigma(e)>0$.
        
        Since the coordinates of $p_{\varepsilon}$ continuously change in $\varepsilon$,
        $p_{\varepsilon}$ is injective if  $p'=p_0$ is injective and $\varepsilon$ is small.
        \end{proof}

For a multigraph $G$, an {\em (open) ear} is a path $P$ of length at least one such that the endvertices of $P$ are distinct vertices of $G$ and the internal vertices are disjoint from $G$. 
Note that a single-edge path between two vertices of $G$ is also an ear. 
Ears appear when constructing 2-connected multigraphs: it is an elementary fact from graph theory  that any 2-connected multigraph can be constructed from any 2-connected subgraph by attaching ears sequentially keeping 2-connectivity. 

The following lemma gives a slightly stronger statement than Lemma~\ref{lem:edge_removal}.

\begin{lemma}\label{lem:ear}
Let $G, H$ be connected multigraphs and suppose that $G$ is obtained from $H$ by attaching an ear $P$.
If $H$ has a $d$-dimensional super stable (injective) tensegrity realization, then so does $G$.
\end{lemma}
\begin{proof}
If $P$ has length one, then attaching $P$ is equivalent to adding an edge, so Lemma~\ref{lem:edge_removal} can be applied.

Suppose the length of $P$ is more than one.
Then attaching the ear $P$ is equivalent to adding a new edge and subdividing the new edge by new vertices.
It is well-known that, in a super stable tensegrity, 
the subdivision of a cable by inserting a new mid-point preserves super stability.
Hence, a super stable (injective) realization of $G$ can be obtained from that of $H$ by first adding an edge as a cable, which is possible by Lemma~\ref{lem:edge_removal}, and then subdividing the new cable.  
\end{proof}


\subsection{Edge contraction} \label{subsec:contraction}
The next goal is to derive the counterpart of Lemma~\ref{lem:edge_removal} for edge contraction.
However, for edge contraction, the injectivity may not be guaranteed in general.
To see this, consider for example a multigraph $C_4^=-e$ obtained from $C_4^=$ by removing an edge $e$, 
and let $e'$ be the edge parallel to $e$ in $C_4^=$.
It is not difficult to check that any injective two-dimensional realization of $C_4^=-e$ is not super stable 
whereas $(C_4^=-e)/e'$ admits an injective two-dimensional super stable realization  since $(C_4^=-e)/e'$ is isomorphic to $K_3^=$.

It seems difficult to characterize the property of having an injective super stable realization completely, so our focus here is to give a sufficient condition.
A key property is the following new concept of stress splittability,
which is inspired by stress non-degeneracy in \cite{JT}.

Let $(G,\sigma,p)$ be a tensegrity.
An equilibrium stress $\omega:E(G)\rightarrow \mathbb{R}$ is said to be {\em splittable} at  a vertex $v\in V(G)$ if there is a proper nonempty subset $F\subsetneq E_{G}(v)$  such that 
\begin{equation}\label{eq:split}
\sum_{f=uv\in F} \omega(f)(p(u)-p(v))=0.
\end{equation}
An equilibrium stress $\omega$ is said to be {\em non-splittable} if it is not splittable at any $v\in V(G)$.
We first remark the following easy observation.
\begin{lemma}\label{lem:non-splittable}
Suppose a $d$-dimensional super stable tensegrity $(G,\sigma, p)$ has a non-splittable equilibrium stress.
Then it has an equilibrium stress $\omega$ such that 
$\omega$ is non-splittable
and $L_{G,\omega}$ is positive semidefinite with nullity $d+1$.
\end{lemma}
\begin{proof}
Since the tensegrity is super stable, by the stress condition for super stability, it has an equilibrium stress $\omega_1$ such that $L_{G,\omega_1}$ is positive semidefinite with nullity $d+1$.
Let $\omega_2$ be a non-splittable equilibrium stress,
and let $\omega=C\omega_1+\omega_2$ for $C \in \mathbb{R}$.
As $L_{G,\omega}=CL_{G,\omega_1}+ L_{G,\omega_2}$
and $\ker L_{G,\omega_2}\subseteq \ker L_{G,\omega_1}$ hold, for sufficiently large $C$, the matrix $L_{G,\omega}$ is positive semidefinite with nullity $d+1$.
For a given point configuration $p$, the set of all non-splittable equilibrium stresses are the complement of finite number of hyperplanes defined by the equation~(\ref{eq:split}) in the space of equilibrium stresses. Since $\omega$ is non-splittable for $C=0$, $\omega$ is non-splittable for a generic choice of $C$.
Thus for a sufficiently large and generic number $C$, $\omega$ has the desired property.
\end{proof}

We are now ready to prove our main technical lemma.
A vertex of a connected multigraph is said to be {\em pendant} if 
it is adjacent to exactly one vertex.
\begin{lemma}\label{lem:edge_contraction}
Let $G$ be a connected multigraph and $e$ be an edge of $G$.
Suppose $G/e$ has a $d$-dimensional super stable tensegrity realization.
Then $G$ has a $d$-dimensional super stable tensegrity realization.

Moreover, if no endvertex of $e$ is pendant  and the realization of $G/e$ is injective and has a non-splittable equilibrium  stress, then the realization of $G$ can be injective with a non-splittable equilibrium stress.
\end{lemma}
\begin{proof}
We first consider the case when there is a parallel edge $e'$ to $e$ in $G$
and show that this case can be reduced to the case when there is no parallel edge to $e$.
Suppose $e$ and $e'$ are parallel. 
Then $(G-e')/e=G/e$.
So $G-e'$ also satisfies the condition of the lemma,
and by induction on the number of edges, the conclusion of the lemma holds for $G-e'$.
Then the statement for $G$ is immediate except for the non-splittability of the stress.
For this, construct a stress $\omega'$ of $G$ from a stress $\omega$ of $G-e'$ by distributing $\omega(e)$ generically to $e$ and $e'$.
Then by the similar argument as in the proof Lemma~\ref{lem:non-splittable}, $\omega'$ is non-splittable if $\omega$ is non-splittable.

Thus we may assume that there is no parallel edge to $e$ in $G$.
Let $n=|V(G)|$. Also, let $e=v_0v_1$, $X=N_G(v_0)\setminus \{v_1\}$, $Y=V\setminus (X\cup \{v_0,v_1\})$, and $E_{0}=E_G(v_0)$ for simplicity of description.
        We simply denote by $v_1$ the vertex in $G/e$ after the contraction of  $e=v_0v_1$.
        Each edge $f\in E_{0}\setminus \{e\}$ in $G$ remains in $G/e$ after the contraction and we keep using the same notation $f$ to denote the edge after the contraction. 
        By this convention, we have $E(G/e)=E(G)\setminus \{e\}$.
        
        The proof strategy goes as follows. By the lemma assumption, we know that there is a Laplacian matrix $L_{G/e}$ of $G/e$ that certifies the super stability of a realization of $G/e$, and our goal is to construct a certificate for $G$ from $L_{G/e}$. Let $G'$ be the graph obtained from $G/e$ by adding all edges between vertices of $X$. We will see that, if we take the Schur complement  of a Laplacian matrix of $G$ at the diagonal entry indexed by $v_0$, then the resulting matrix is a Laplacian matrix of $G'$.
        Hence, in view of Proposition~\ref{prop:schur rank}, we may focus on constructing a Laplacian matrix of $G'$ which is the Shur complement of some Laplacian of $G$.   The idea is to construct such a Laplacian $L_{G'}$ of $G'$ by a perturbation of $L_{G/e}$ and then recover a Laplacian of $G$ from $L_{G'}$ by reversing the process of taking the Schur complement. Not all Laplacian matrices of $G'$ are the Schur complement of Laplacian matrices of $G$, so we need to understand which Laplacian matrix of $G'$ can be the Schur complement of a Laplacian of $G$.
        
        To see this, consider any edge weight $w:E(G)\rightarrow \mathbb{R}$ and $L_{G,\omega}\in {\cal L}(G)$ such that the diagonal entry indexed by $v_0$ is nonzero.
        For simplicity of description, for each subgraph $H$, we use $L_{H,\omega}$ to denote $L_{H,\omega_{|H}}$, where $\omega_{|H}$ is the restriction of $\omega$ to $E(H)$.
        We first split $L_{G,\omega}$ into two matrices such that all the nonzero entries supported by edges incident to $v_0$ are given in the second matrix as follows:
        \begin{align}
        L_{G,\omega}&=L_{G-E_{0},\omega}+L_{E_{0},\omega}.
        \label{eq:contraction-1}
         \end{align}
         (Recall that $E_0$ denotes the set of edges incident to $v_0$ in $G$.)
        We then introduce notation for entries of the second matrix $L_{E_0,\omega}$ as follows:
        \begin{align}
        L_{G,\omega}&=L_{G-E_{0},\omega}+ \kbordermatrix{ &v_0& v_1 & X  & Y \\
        v_0 & \varepsilon^{-1} & -\varepsilon^{-1}- s & \omega_X^{\top} &   0  \\
        v_1 & -\varepsilon^{-1} - s   & \varepsilon^{-1}+s & 0 & 0 \\ 
        X &  \omega_X & 0 & -{\rm diag}(\omega_X) & 0 \\
        Y & 0 & 0 & 0 & 0 },\label{eq:contraction0}
         \end{align}
        where $\varepsilon^{-1}$ denotes $L_{G,\omega}[v_0,v_0]$, 
        $\omega_X \in \mathbb{R}^{X}$ is the vector obtained by arranging 
        the edge weight between $u$ and $v_0$ over $u\in X$,  and $s=\sum_{u\in X} \omega_X(u)$.
        (Here $s$ is defined to be $\sum_{u\in X} \omega_X(u)$  since the sum of the entries of the first column (or the first row) in the second matrix must be zero.)

        If we take the Schur complement of (\ref{eq:contraction0}) at the left-top corner, then the resulting matrix $L_{G,\omega}/v_0$ becomes
        \begin{equation}\label{eq:contraction2}
        L_{G,\omega}/v_0= L_{G-v_0,\omega}+ \kbordermatrix{ & v_1 & X  &  Y \\
        v_1 & -s-\varepsilon s^2   & \omega_X^{\top}+\varepsilon s \omega_X^{\top} & 0  \\ 
        X &  \omega_X +\varepsilon s \omega_X  & -{\rm diag}(\omega_X) -\varepsilon \omega_X\omega_X^{\top}& 0 \\
        Y & 0 & 0 & 0  }
         \end{equation}
        where $L_{G-v_0,\omega}\in {\cal L}(G-v_0)$ is obtained from $L_{G-E_0,\omega}$ by removing the top zero row and the left zero column.
        This can be further written as 
        \begin{align}
        \nonumber
            L_{G,\omega}/v_0 &= 
            L_{G-v_0,\omega}+ \kbordermatrix{ & v_1 & X  &  Y \\
        v_1 & -s   & \omega_X^{\top} & 0  \\ 
        X &  \omega_X   & -{\rm diag}(\omega_X) & 0 \\
        Y & 0 & 0 & 0  }+ \varepsilon \left(\kbordermatrix{ & v_1 & X  & Y  \\
        v_1 & -s^2   & s \omega_X^{\top} & 0 \\ 
        X &   s \omega_X & -\omega_X\omega_X^{\top}  & 0 \\
        Y & 0 & 0 & 0 }\right)\\
        \label{eq:contraction3}
        &=L_{G/e,\omega}+ \varepsilon \left(\kbordermatrix{ & v_1 & X  & Y  \\
        v_1 & -s^2   & s \omega_X^{\top} & 0 \\ 
        X &   s \omega_X & -\omega_X\omega_X^{\top}  & 0 \\
        Y & 0 & 0 & 0 }\right),
         \end{align}
         where $L_{G/e,\omega}$ is the Laplacian of $G/e$ weighted by  $\omega_{|G-v_0v_1}$. 
         (Recall our convention that $E(G/e)=E(G)\setminus \{v_0v_1\}$).
         Note that $L_{G,\omega}/v_0\in {\cal L}(G')$
         and (\ref{eq:contraction3})  is exactly the form we are looking at.
         
         With this relation in mind, we consider the following set in $\mathcal {L}^n\times \mathbb{R}$.
         \[
         {\cal M}:=\left\{ 
         \left(L_{G/e,\omega}+ \varepsilon \left(\kbordermatrix{ & v_1 & X  & Y \\
        v_1 & -s^2   & s \omega_X^{\top} & 0 \\ 
        X &   s \omega_X & -\omega_X\omega_X^{\top}  & 0 \\
        Y & 0 & 0 & 0 }\right),\varepsilon \right) :
         \begin{array}{l} \omega:E(G)\setminus \{v_0v_1\} \rightarrow \mathbb{R}  
         \\   \varepsilon \in \mathbb{R}
         \end{array}
         \right\}
         \]
         and let ${\cal M}_0$ be the subset of ${\cal M}$ consisting of elements with $\varepsilon=0$. 
         Observe that the first argument in any element in ${\cal M}$ is of the form (\ref{eq:contraction3}).
         
           Let $(G/e,\sigma,p)$ be a $d$-dimensional super stable tensegrity realization,
   and let $L_{G/e,\omega^*}$ be a weighted Laplacian matrix of $G/e$ that certifies the super stability of $(G/e,\sigma,p)$.
   By applying an affine transformation, we may suppose that $p$ is a reduced kernel representation of $L_{G/e,\omega^*}$.
        Then, $(L_{G/e,\omega^*},0)\in {\cal M}_0$. 
        Observe also that, for any $(L,0)\in {\cal M}_0$, $T_{(L,0)} {\cal M}_0=T_L {\cal L}(G/e)\times \{0\}$ holds for the tangent space.
        Hence, by Proposition~\ref{prop:conic}, the Euclidean SAP of $L_{G/e,\omega^*}$ implies that ${\cal M}_0$ and ${\cal L}_{(n-1)-(d+1)}^{n-1} \times \mathbb{R}$ intersect transversally at $(L_{G/e,\omega^*}, 0)$.
        By Proposition~\ref{prop:transversal}, there is a positive number $\overline{\varepsilon}$ and a smooth function $\gamma:[-\overline{\varepsilon}, \overline{\varepsilon}] \rightarrow \mathcal{L}^{n-1}\times \mathbb{R}$ such that 
        $\gamma(0)=(L_{G/e,\omega^*},0)$ and $\gamma(\varepsilon)\in {\cal M} \cap ({\cal L}_{(n-1)-(d+1)}^{n-1}\times \mathbb{R})$ for $\varepsilon \in [-\overline{\varepsilon}, \overline{\varepsilon}]$.
        Denote $\gamma(\varepsilon)=(L_{\varepsilon}, \varepsilon)$.
        Since the entries of $L_{\varepsilon}$ continuously change with respect to $\varepsilon$, 
         we have $L_{\varepsilon}\in {\cal L}_{(n-1)-(d+1)}^{n-1}$ and $L_{\varepsilon}\succeq 0$ if $\varepsilon$ is sufficiently small.
         
         By the definition of $\cal M$, there is an edge weight $\omega_{\varepsilon}: E(G)\setminus \{v_0v_1\}\rightarrow \mathbb{R}$ such that $L_{\varepsilon}$ is given in the form of (\ref{eq:contraction3}) with respect to  $\omega_{\varepsilon}$ and $\varepsilon$. 
         By the continuity of $L_{\varepsilon}$ with respect to $\varepsilon$ and $L_0=L_{G/e,\omega^*}$, we can take $\omega_{\varepsilon}$ continuously with respect to $\varepsilon$ such that $\omega_0=\omega^*$.
         
         For $\varepsilon>0$, we extend $\omega_{\varepsilon}$, which is defined over $E(G)\setminus \{v_0v_1\}$, to an edge weight  $\omega_{\varepsilon}'$ of $E(G)$ 
         by setting
         \begin{equation}\label{eq:bridging_stress}
         \omega_{\varepsilon}'(v_0v_1) = \varepsilon^{-1}+s_{\varepsilon},
         \end{equation}
         where $s_{\varepsilon}=\sum_{u\in X} \omega_{\varepsilon, X}(u)$.
         Since  $L_{\varepsilon}$ is of the form (\ref{eq:contraction3}) with respect to $\omega_{\varepsilon}$, $L_{G,\omega_{\varepsilon}'}$ is the Laplacian matrix of $G$ obtained  from $L_{\varepsilon}$ by reversing the above process, i.e., 
         $L_{G,\omega_{\varepsilon}'}/v_0=L_{\varepsilon}$.
         (The definition (\ref{eq:bridging_stress}) is coming from the $ v_0v_1$-th entry of the matrix in (\ref{eq:contraction0}).)
         
         As $\varepsilon>0$ and $L_{\varepsilon}\succeq 0$, by Proposition~\ref{prop:schur rank},
         $L_{G,\omega_{\varepsilon}'}$ is positive semidefinite and its nullity is equal to that of $L_{\varepsilon}$.
         Therefore,  $L_{G,\omega_{\varepsilon}'} \in {\cal L}^*(G)\cap {\cal L}_{n-(d+1)}^n$ and $L_{G,\omega_{\varepsilon}'} \succeq 0$ if $\varepsilon$ is a sufficiently small nonzero number.

         We now construct a realization of $G$.
         Since the entries of $L_{\varepsilon}$ continuously change with respect to $\varepsilon$, we can take a smooth function 
        $[-\overline{\varepsilon}, \overline{\varepsilon}]\ni \varepsilon \mapsto p_{\varepsilon}\in (\mathbb{R}^d)^{n-1}$ such that 
        $p_0=p$ and $p_{\varepsilon}$ is a reduced kernel representation of $L_{\varepsilon}$.
        For $\varepsilon>0$, we can apply Proposition~\ref{prop:schur} to $p_{\varepsilon}$ to get an extension $p_{\varepsilon}'$ of $p_{\varepsilon}$ such that $p_{\varepsilon}'$ is a reduced kernel representation of $L_{G,\omega_{\varepsilon}'}$.
        Specifically, by using the formula of Proposition~\ref{prop:schur} and $s_{\varepsilon}=\sum_{u\in X} \omega_{\varepsilon, X}(u)$, we have 
        \begin{align}\label{eq:explicite_form}
        p_{\varepsilon}'(v_0)&=p_{\varepsilon}(v_1)
        -\varepsilon\left( \sum_{f=v_0u\in E_{0}}\omega_{\varepsilon}(f)(p_{\varepsilon}(u)-p_{\varepsilon}(v_1))\right).
        \end{align}
        Since $p_{\varepsilon}'$ is a reduced kernel representation of $L_{G,\omega_{\varepsilon}'}$, 
        $L_{G,\omega_{\varepsilon}'}$ certifies the stress condition for super stability of tensegrity $(G,\sigma',p_{\varepsilon}')$ (by taking an appropriate sign function $\sigma'$).
        
        We check the conic condition of $(G,\sigma',p_{\varepsilon}')$.
        Since $p_{\varepsilon}'$ is an extension of $p_{\varepsilon}$
        and $p_{\varepsilon}'(v_0)$ is given by (\ref{eq:explicite_form}), 
        $p_0'$ is well-defined by letting $\varepsilon \rightarrow 0$,
        and we have 
        \begin{align*}
        p_{0}'(u)&=p_{0}(u)=p(u) \qquad (u\in V(G)\setminus \{v_0\}) \\ p_0'(v_0)&=p_0(v_1)=p(v_1).
        \end{align*}
Therefore, since $(G/e,p)$ satisfies the conic condition (as $(G/e, \sigma,p)$ is super stable), 
$(G,p_0')$ satisfies the conic condition.
By continuity of $p_{\varepsilon}'$ and Proposition~\ref{prop:conic_generic}, $(G,p_{\varepsilon}')$ also satisfies the conic condition if $\varepsilon$ is sufficiently small.

In total, $(G,\sigma',p_{\varepsilon}')$ is a $d$-dimensional super stable tensegrity realization of $G$.

Finally we check that $(G,\sigma',p_{\varepsilon}')$ is injective
and admits a non-splittable stress if so does $(G/e,\sigma,p)$.
By Lemma~\ref{lem:non-splittable}, we may suppose that the initial stress $\omega^*:E(G/e)\rightarrow \mathbb{R}$ of $(G/e,\sigma,p)$ is non-splittable.
If $(G,\sigma',p_{\varepsilon}')$ is not injective for any small $\varepsilon$,
then $p_{\varepsilon}'(v_0)=p_{\varepsilon}'(v_1)$ by the injectivity of $(G/e,\sigma,p)$.
By (\ref{eq:explicite_form}), we get $\displaystyle \sum_{f=v_0u\in E_{0}}\omega_{\varepsilon}(f)(p_{\varepsilon}(u)-p_{\varepsilon}(v_1))=0$.
By continuity, this in turn implies 
$\displaystyle \sum_{f=v_0u\in E_{0}}\omega_{0}(f)(p_{0}(u)-p_{0}(v_1))=0$, and equivalently 
\begin{equation}\label{eq:E_0}
\sum_{f=v_0u\in E_{0}}\omega^*(f)(p(u)-p(v_1))=0
\end{equation}
by $\omega_0=\omega^*$ and $p_0=p$.
However, since $e=v_0v_1$ is not incident to a degree-one vertex, $E_0$ is a nonempty proper subset of $E_{G/e}(v_1)$, and 
(\ref{eq:E_0}) implies that $\omega^*$ is splittable at $v_1$ in $G/e$, a contradiction.
Thus, $(G,\sigma',p_{\varepsilon}')$ is injective.

Finally, we prove that the resulting stress $\omega_{\varepsilon}'$ of $(G,\sigma',p_{\varepsilon}')$ is non-splittable.
Suppose for a contradiction that $\omega_{\varepsilon}'$ is splittable for any small $\varepsilon$.
The non-splittability of $\omega^*$ and $\omega_0=\omega^*$ imply that 
$\omega_{\varepsilon}$ is non-splittable in $G/e$ if $\varepsilon$ is sufficiently small.
Since $\omega_{\varepsilon}'$ is an extension of $\omega_{\varepsilon}$,
$\omega_{\varepsilon}'$ must be splittable at $v_i$ for some $i\in \{0,1\}$.
So 
\begin{equation}\label{eq:splittable_at_v_i}
\sum_{f=uv_i\in F}\omega_{\varepsilon}'(f)(p_{\varepsilon}'(u)-p_{\varepsilon}'(v_i))=0
\end{equation}
for some nonempty $F\subsetneq E_G(v_i)$.
Since $\omega_{\varepsilon}'$ is an equilibrium stress of $(G,\sigma',p_{\varepsilon}')$, (\ref{eq:splittable_at_v_i}) still holds by replacing $F$ with $E_G(v_i)\setminus F$.
Thus, we may always suppose that $v_0v_1\notin F$.
Then, by letting $\varepsilon\rightarrow 0$ in (\ref{eq:splittable_at_v_i}), we obtain
\[
0=\sum_{f=uv_i\in F}\omega_{0}'(f)(p_{0}'(u)-p_{0}'(v_i))=\sum_{f=uv_1\in F}\omega^*(f)(p(u)-p(v_1)).
\]
Since $F\subsetneq E_{G/e}(v_1)$,
this implies that $\omega^*$ is splittable at $v_1$ in $G/e$, a contradiction.
This contradiction completes the proof that $\omega_{\varepsilon}'$ is non-splittable.
\end{proof}

\subsection{Minor Monotonicity} \label{subsec:minor}
The super stability enjoys a minor monotone property in the following sense.
\begin{theorem}\label{thm:minor_monotone}
Let $G$ be a connected multigraph  and $H$ be a minor of $G$. 
If $H$ admits a $d$-dimensional super stable tensegrity realization, then 
$G$ admits a $d$-dimensional super stable tensegrity realization.
\end{theorem}
\begin{proof}
We may suppose that $H$ is connected.
Since $G$ is connected and $H$ is a minor of $G$,
$H$ can be obtained from $G$ by contracting edges and removing non-bridge edges.
Hence the statement follows from Lemmas~\ref{lem:edge_removal} and~\ref{lem:edge_contraction}.
\end{proof}

\begin{corollary}\label{cor:monotonicity_lambda}
Let $G$ and $H$ be multigraphs.
If $H$ is a minor of $G$, then $\lambda(H)\leq \lambda(G)$. 
\end{corollary}

As remarked above, the analogous statement for injective realizations dose not hold. The following is what we can currently achieve by Lemmas~\ref{lem:ear} and~\ref{lem:edge_contraction}.

\begin{theorem}\label{thm:injective}
Let $G$ be a $2$-connected multigraph.
Let $H$ be a $2$-connected multigraph which is a minor of $G$. 
If $H$ admits a $d$-dimensional injective super stable tensegrity realization 
which has a non-splittable equilibrium stress, then $G$ admits a $d$-dimensional injective super stable tensegrity realization.
\end{theorem}
\begin{proof}
We show that $G$ can be constructed from $H$ by first applying vertex splitting operations and then attaching ears such that each intermediate multigraph is 2-connected.
If this is the case, then the theorem follows by first applying Lemma~\ref{lem:edge_contraction} and then Lemma~\ref{lem:ear}.

Since $H$ is a minor of $G$, $G$ has a subgraph $G'$ such that $V(G')$ admits a partition  $\{V_x:x \in V(H)\}$ such that each $V_x$ is associated with a vertex $x$ of $H$,
each $V_x$ induces a connected subgraph in $G'$, and $G'$ has  exactly $k$ edges between $V_x$ and $V_y$ if $H$ has $k$ edges between $x$ and $y$~(see, e.g.,\cite{D}).
Then one can obtain $H$ from $G'$ by contracting each $V_x$ into a vertex for all $V_x$.

We take $G'$ as small as possible. Then, for each $x \in V(H)$, $V_x$ induces a tree in $G'$ and every leaf of the tree is adjacent to a vertex in a different set $V_y$.
With this property, the 2-connectivity of $H$ implies the 2-connectivity of $G'$. Since this property is preserved by a contraction operation of an edge in the tree induced by $V_x$,  the 2-connectivity of $H$ also implies the 2-connectivity of any intermediate multigraph in the process of constructing $H$ from $G'$. Thus, $G'$ can be constructed from $H$ by  vertex splitting operations such that  each intermediate multigraph is 2-connected.

Now, $G'$ is 2-connected, so the 2-connectivity of $G$ implies that $G$ can be constructed from $G'$ by attaching ears keeping 2-connectivity.
Thus, $G$ admits a required construction. 
\end{proof}
To check whether a tensegrity realization of a graph $H$ admits a non-splittable stress, the following sufficient condition is useful.
\begin{lemma} \label{lem:split deg}
Let $(G,\sigma,p)$ be a tensegrity.
Suppose that the dimension of the affine span of $\{p(u):u \in N_G(v)\cup \{v\}\}$ is at least $|E_G(v)| -1$ for all $v \in V(G)$. Then every strictly proper equilibrium stress of $(G,\sigma,p)$ is non-splittable.
\end{lemma}
\begin{proof}
Suppose on the contrary that there is a strictly proper equilibrium stress $\omega$ which is splittable at a vertex $v \in V(G)$.
Then there is a proper nonempty subset $F \subseteq E_G(v)$ satisfying the equation~(\ref{eq:split}).
Since $\omega$ is strictly proper, the equation~(\ref{eq:split}) implies  $\{p(u)-p(v): f=uv \in F\}$ is linearly dependent.
Since $\omega$ is an equilibrium stress of $(G,\sigma,p)$, $E_G(v) \setminus F$ also satisfies the equation~(\ref{eq:split}), and hence $\{p(u)-p(v): f=uv \in E_G(v) \setminus F\}$ is linearly dependent.
Hence the dimension of the linear span of $\{p(u)-p(v): f=uv \in E_G(v)\}$ is at most $|E_G(v)|-2$, which means the affine span of $\{p(u): u\in N_G(v)\cup \{v\}\}$ has dimension at most $|E_G(v)|-2$,  a contradiction.
\end{proof}





We are now ready to prove Theorem~\ref{thm:Kt}
\begin{proof}[Proof of Theorem~\ref{thm:Kt}]
Consider the $d$-dimensional tensegrity realization $(K_{d+2},\sigma,p)$ of $K_{d+2}$ given in Example 3.
In this realization, for each $v$, the affine dimension of $\{p(u): u\in N_{K_{d+2}}(v)\cup \{v\}\}$ is $d$.
Hence, by Lemma~\ref{lem:split deg}, the equilibrium stress $\omega$ given in Example 3 is non-splittable.
So the statement follows from Theorem~\ref{thm:injective} by setting $H=K_{d+2}$.
\end{proof}

\section{Characterizing Multigraphs of Super Stable Tensegrities}\label{sec:characterization_lambda}
In this section we shall give a characterization of multigraphs which can be realized as three-dimensional super stable tensegrities. 
This can be done by establishing an exact relation between $\lambda$ and the Colin de Verdi\`{e}re number $\nu$.

This section is organized as follows.
In Subsection~\ref{subsec:colin} we give a formal definition of the Colin de Verdi\`{e}re number $\nu$.
In Subsection~\ref{subsec:coning} we introduce the coning operation, which is the basic tool for linking $\lambda$ and $\nu$.
In Subsection~\ref{subsec:slicing} we explain the slicing and sliding theorem by Connelly, Gortler and Theran, which is the key ingredient for the exact relation between $\lambda$ and $\nu$.
In Subsection~\ref{subsec:characterization}, we put all observations together 
and derive the combinatorial characterization as a corollary.

\subsection{Colin de Verdi\`{e}re number $\nu$}\label{subsec:colin}
Roughly speaking, the Colin de Verdi\`{e}re number $\nu$ can be defined by replacing Laplacian matrices with adjacency matrices in (\ref{eq:def_L}).
For a formal definition, we need some notation.
For a multigraph $G$, let
\begin{align*}
{\cal A}^*(G)=\left\{A\in {\cal S}^n \mid \begin{array}{ll} 
A[i,j]= 0 & \text{if $G$ has no  edge between $i$ and $j$ with $i\neq j$} \\
A[i,j]\neq 0 & \text{if $G$ has a single edge between $i$ and $j$} \\
A[i,j]\in  \mathbb{R} & \text{if $G$ has parallel edges between $i$ and $j$ or $i=j$}
\end{array} \right\},
\end{align*}
where ${\cal S}^n$ denotes the set of $n\times n$ symmetric matrices.
The set of $n\times n$ symmetric matrices of rank $k$ is denoted by ${\cal S}^n_k$.




We say that a matrix $A\in {\cal A}^*(G)$ satisfies the {\em Strong Arnold Property (SAP)} if there is no nonzero $X\in {\cal S}^n$ satisfying $X A=0$ and $\langle X, \bm{e}_i \bm{e}_j^{\top}\rangle=0$ for every $i,j$ with $i=j$ or $ij\in E(G)$.
Observe the similarity between the SAP and the Euclidean SAP introduced in Proposition~\ref{prop:conic}.

For a multigraph $G$, the {\em Colin de Verdi\`{e}re number} $\nu$ is defined by
\begin{equation}\label{eq:nu}
        \nu(G)=\max\left\{\dim \ker A: \begin{array}{l} A\in {\cal A}^*(G) \\ A\succeq 0 \\ \text{$A$ satisfies the SAP} \end{array}\right\}.
\end{equation}
This graph parameter was first introduced by Coin de Verdi\`{e}re~\cite{C98} for connected simple graphs and extended to multigraphs by van der Holst~\cite{H96,H02}.
Our idea of using multigraphs to investigate the combinatorics of tensegrities was inspired by the work of van der Holst.
Note that Coin de Verdi\`{e}re~\cite{C90} also introduced a different graph parameter $\mu$, which is more widely recognized.
The Colin de Verdi\`{e}re number $\nu$ is also referred to as the algebraic width in \cite{L}. 

\subsection{Coning: bridging $\lambda$ and $\nu$}\label{subsec:coning}
A key tool to bridge between $\lambda$ and $\nu$ is the so-called coning operation. 
A connection between $\lambda$ and $\nu$ through coning was implicit in the work by Laurent and Varvitsiotis~\cite{LV}, and we will make it explicit in this subsection.

Let $G$ be a multigraph.
The multigraph obtained from $G$ by adding a new vertex $v_0$ and new parallel edges between $v_0$ and each vertex in $G$ is called the {\em cone} of $G$, and denoted it by $\nabla G$.
The new vertex $v_0$ is called the {\em cone vertex}.

The following lemma relates $\lambda$ with $\nu$. A slightly weaker statement  can be found in \cite[Lemma 4.11]{LV}.
\begin{lemma}\label{lem:coning_nu}
For any multigraph $G$, $\lambda(\nabla G)=\nu(G)$.
\end{lemma}
\begin{proof}
Let $d=\lambda(\nabla G)$, and let $L$ be a maximizer of (\ref{eq:def_L}) for  $\nabla G$. Then $\dim \ker L=d+1$ and $L\succeq 0$.
Let $A$ be the principal  submatrix of $L$ indexed by $V(G)$.
Then $A\succeq 0$ and $A\in {\cal A}^*(G)$.
We show that $A$ certifies $\nu(G)\geq d$.
To see this,  take a reduced kernel representation $p$ of $L$.
We may assume $p(v_0)=0$, where $v_0$ is the cone vertex of  $\nabla G$.
Then $p_{|V(G)}$ is a kernel representation of $A$,
and  $\dim \ker A= d$ follows.
Observe that the SAP of $A$ is equivalent to the Euclidean SAP of $L$.
This is because,  $A \in {\cal A}^*(G)$ satisfies the SAP if and only if there is no nonzero symmetric matrix $S \in \mathcal{S}^{d}$ satisfying $p(i)^\top S p(i)=0$ ($i\in V(G)$) and $p(i)^\top S p(j)$ ($ij \in E(G)$) (see~\cite[Theorem 4.2]{H02}), and the latter condition is equivalent to Proposition~\ref{prop:conic} (iii) as $p(v_0)=0$.
Thus, $A$ satisfies the SAP, and $\nu(G)\geq d$ follows.
 
 Conversely, let $d=\nu(G)$, and let $A$ be a maximizer of (\ref{eq:nu})  for $\nabla G$. Then $\dim \ker A=d$ and $A\succeq 0$.
 Let 
 \[
 L=\begin{pmatrix} \mathbf{1}^{\top} A \mathbf{1} & -\mathbf{1}^{\top} A \\
 -A \mathbf{1} & A \end{pmatrix}.
 \]
 Then $L\in \cL^*(\nabla G)$.
 Moreover, we have $L=\begin{pmatrix} \mathbf{1} & -I_n \end{pmatrix}^{\top} A \begin{pmatrix} \mathbf{1} & -I_n \end{pmatrix}$,
 implying that $L\succeq 0$ and $\dim \ker L=d+1$.
 The Euclidean SAP of $L$ can be checked by the same reason as the first case. 
 Thus,  $\lambda(\nabla G)\geq d$ follows.
\end{proof}

In view of Lemma~\ref{lem:coning_nu}, the next question is to understand a relation between $\lambda(G)$ and $\lambda(\nabla G)$.
The following relation is not difficult to check.

\begin{lemma}\label{lem:coning_lambda}
For any multigraph $G$, $\lambda(\nabla G)\geq \lambda(G)+1$.
\end{lemma}
\begin{proof}
This follows from the fact that the coning operation preserves 
the super stability~\cite[Theorem 4.6]{CGT}.
Alternatively, one can directly check that, for  
a maximizer $L$ of (\ref{eq:def_L}) for $G$ with $d=\dim \ker L$, 
the matrix $\begin{pmatrix} 0 & 0 \\ 0 & L\end{pmatrix}$
certifies $\lambda(\nabla G)\geq d+1$.
\end{proof}

\subsection{Slicing and Sliding}\label{subsec:slicing}
By Lemmas~\ref{lem:coning_nu} and~\ref{lem:coning_lambda}, it follows that $\lambda(G)\leq \lambda(\nabla G)-1=\nu(G)-1$.
The key ingredient to prove the converse relation is an observation by Connelly, Gortler, and Theran~\cite{CGT} that the super stability  is preserved by sliding and slicing operations of coned bar-joint frameworks.

Since we have to take care of the existence of proper stresses in the tensegrity case, let us look at the construction of  Connelly, Gortler, and Theran for completeness. 
For a given tensegrity realization $(\nabla G,\sigma, p)$ of $\nabla G$ with the cone vertex at the origin, suppose that 
\begin{equation}\label{eq:assumption}
\text{no point of $(\nabla G, \sigma, p)$ except the cone vertex lies at the origin. }
\end{equation}
A {\em sliding} of $(\nabla G, \sigma, p)$ is a tensegrity $(\nabla G,\sigma', q)$ such that 
\begin{itemize}
    \item $q(v)=s_v p(v)$ for some non-zero scalar $s_v$ for each $v\in V(G)$ (and the cone vertex remains at the origin), and 
    \item $\sigma'(uv)={\rm sign}(s_u) \cdot {\rm sign}(s_v) \cdot \sigma(uv)$ for every $e=uv\in E(G)$ and the sign of each edge incident to the cone vertex is unchanged.
\end{itemize}
Suppose $(\nabla G,\sigma',q)$ is obtained from $(\nabla G,\sigma,p)$ by $q(v)=s_v p(v)$ for $v\in V(G)$, 
and suppose $L$ is a Laplacian matrix of $\nabla G$ which certifies the super stability of  $(\nabla G, \sigma, p)$.
Denote the cone vertex by $v_0$ and $V(G)=\{v_1,\dots, v_n\}$, and consider 
\begin{equation}\label{eq:slicing}
L':= 
\begin{pmatrix} 1 & 0 & 0  & \dots &  0 \\ 
1-\frac{1}{s_{v_1}} & \frac{1}{s_{v_1}} & 0 & \dots & 0 \\
\vdots & 0  & \ddots  &  &  \vdots \\
\vdots & \vdots &  & \ddots & 0 \\
1-\frac{1}{s_{v_n}} & 0 & \vdots & 0 & \frac{1}{s_{v_n}} 
\end{pmatrix}^{\top}
L
\begin{pmatrix} 1 & 0 & 0  & \dots &  0 \\ 
1-\frac{1}{s_{v_1}} & \frac{1}{s_{v_1}} & 0 & \dots & 0 \\
\vdots & 0  & \ddots  &  &  \vdots \\
\vdots & \vdots &  & \ddots & 0 \\
1-\frac{1}{s_{v_n}} & 0 & \vdots & 0 & \frac{1}{s_{v_n}} 
\end{pmatrix},
\end{equation}
where we assume that the rows and the columns are indexed in the order $v_0, v_1, \dots, v_n$.
It can be rapidly checked that $L'\in {\cal L}^*(\Delta G)$.
Moreover, by Sylvester's law of inertia, 
$L'$ is positive semidefinite with $\rank L'=\rank L$, so $L'\in {\cal L}_{+,n-(d+1)}^n$.
Hence, $L'$ certifies the stress condition of $(\nabla G, \sigma', q)$.
A nontrivial observation due to Connelly, Gortler, and Theran~\cite{CGT} is that  $(\nabla G, \sigma', q)$ also satisfies the conic condition, and thus $(\nabla G, \sigma', q)$ is super stable.


Let   $(\nabla G, \sigma, p)$ be a $(d+1)$-dimensional coned tensegrity satisfying the assumption (\ref{eq:assumption}) and $H$ be a hyperplane in $\mathbb{R}^{d+1}$ not through the origin.
Then by a sliding operation we can convert $(\nabla G, \sigma, p)$ to a framework $(\nabla G, \sigma', q)$ such that all the points except the cone point lie on the hyperplane $H$. By identifying $H$ with $\mathbb{R}^d$, the sub-tensegrity $(G,q)$ obtained by removing the cone vertex can be regarded as a tensegrity in $\mathbb{R}^d$.  The resulting framework $(G,q)$ is called a {\em slicing} of $(\nabla G, \sigma, p)$.
Since all the non-cone point lie on $H$ in $(\nabla G, \sigma', q)$, the (net) stress $\omega'$ of $(\nabla G, \sigma', q)$ obtained in the construction  (\ref{eq:slicing})  is zero along each coned edge. 
This implies that the restriction of $\omega'$ to $E(G)$ is also an equilibrium stress of $(G, \sigma', q)$. 
Connelly, Gortler, and Theran showed that the resulting stress certifies the super stability of $(G, \sigma', q)$. 
Rephrasing the results explained so far in our terminology we have the following.

\begin{theorem}[Connelly, Gortler, and Theran (adapted)]\label{thm:slicing} 
Let $G$ be a multigraph.
Suppose $\nabla G$ has a $(d+1)$-dimensional super stable tensegrity realization such that the cone vertex does not coincide with other points in the realization.
Then $G$ has a $d$-dimensional super stable tensegrity realization.
\end{theorem}
The condition on the realization of $\nabla G$ in the statement of Theorem~\ref{thm:slicing} is due to the assumption (\ref{eq:assumption}). We now discuss how to deal with this assumption.

\begin{lemma}\label{lem:coned_coincident}
Let $(\nabla G, \sigma, p)$ be a  $(d+1)$-dimensional super stable tensegrity with the cone vertex $v_0$,
and let $X=\{v\in V(G): p(v)=p(v_0)\}$.
Then the tensegrity obtained from $(\nabla G, \sigma, p)$ by removing $X$ is super stable.
\end{lemma}
\begin{proof}
Let $v\in X$. 
We first show that  $(\nabla G-v,\sigma, p)$ is super stable.
Let $L$ be a stress matrix of $(\nabla G,\sigma, p)$  that certifies the super stability of $(\nabla G,\sigma, p)$.
If the $v$-th diagonal entry of $\Omega$ is zero, then the $v$-th row and column are zero by the positive semidefiniteness of $L$.
Then the kernel of  $L$  contains the characteristic vector of $v$, implying that $p(v)$ is not at the origin, a contradiction.

Hence, the $v$-th diagonal entry of $L$ is positive.
Suppose $L$ is written as $\begin{pmatrix} a & b^{\top} \\ b & L'\end{pmatrix}$, where we assume the first row/column is indexed by $v$. Then $a>0$.
We consider the Schur complement $L/v$ at the $v$-th diagonal,
which is $L/v=L' -\frac{1}{a} b b^{\top}$. 
Since $L$ is positive semidefinite, so is $L/v$ and $\rank L/v=\rank L -1$ by Proposition~\ref{prop:schur rank}.
Also the all-one vector is in the kernel of $L/v$
since $L/v \bm{1}=(L'-\frac{1}{a}bb^{\top})\bm{1}=\bm{0}$
by $L'\bm{1}=-b$ and $b^{\top}\bm{1}=-a$.

Let $v_0$ be the cone vertex, $\chi_{v_0}$ be the characteristic vector of $v_0$, and 
\[
L_v:= L/v+\frac{1}{a}(b+a \chi_{v_0})(b+a \chi_{v_0})^{\top}.
\]
Then  $L_v\in {\cal L}^*(\nabla (G-v))$, 
and moreover  it is the Laplacian of $\nabla (G-v)$ weighted by an equilibrium stress of $(\nabla G -v ,\sigma, p)$.
Since $L/v$ is positive semidefinite and $\dim \ker  L/v=\dim \ker  L$, 
$L_v$ is positive semidefinite with $\dim \ker L_v=\dim \ker L/v=\dim \ker  L$.
The conic condition clearly holds in $(\nabla G -v, p)$ since the set of edge directions does not change. 
Thus,  $(\nabla G -v ,\sigma, p)$ is super stable.

We can apply the same argument in the resulting tensegrity until we remove all the vertices  in $X$.
We finally obtain a Laplacian matrix that certifies the  super stability of $(\nabla G-X, \sigma, p)$.
\end{proof}

Combining Theorem~\ref{thm:slicing} and Lemma~\ref{lem:coned_coincident}, we obtain the following.
\begin{lemma}\label{lem:induced}
Let $G$ be a multigraph. Then $G$ has an induced subgraph $H$ such that $\lambda(H)\geq \lambda(\nabla G)-1$.
\end{lemma}
\begin{proof}
Let $d=\lambda(\nabla G)-1$, and let $(\nabla G, \sigma, p)$ be a $(d+1)$-dimensional super stable tensegrity realization.
Let $H$ be a subgraph of $G$ obtained by removing all the vertices in $X=\{v\in V(G): p(v)=p(v_0)\}$.
By Lemma~\ref{lem:coned_coincident}, $(\nabla H, \sigma, p)$ is super stable.
By Theorem~\ref{thm:slicing}, $H$ has a $d$-dimensional super stable tensegrity realization, implying $\lambda(H)\geq d$. 
\end{proof}

\subsection{Characterizing Multigraphs Having Super Stable Tensegrity Realizations} \label{subsec:characterization}
The following theorem is our main observation in this section.
\begin{theorem}\label{thm:equality}
For a multigraph $G$, $\lambda(G)=\nu(G)-1$.
\end{theorem}
\begin{proof}
By Lemma~\ref{lem:coning_nu}, it suffices to show $\lambda(G)=\lambda(\nabla G)-1$.
By Lemma~\ref{lem:coning_lambda}, $\lambda(G)\leq \lambda(\nabla G)-1$ holds. 
Conversely, Corollary~\ref{cor:monotonicity_lambda} and by Lemma~\ref{lem:induced}, $G$ has a subgraph $H$ such that 
$\lambda(G)\geq \lambda(H)\geq \lambda(\nabla G)-1$.
Thus $\lambda(G)=\lambda(\nabla G)-1$ holds.
\end{proof}

Theorem~\ref{thm:characterization_super_stable} stated in the introduction is an immediate corollary of 
Theorem~\ref{thm:equality}. 

\begin{proof}[Proof of Theorem~\ref{thm:characterization_super_stable}]
In view of Theorem~\ref{thm:minor_monotone},
a multigraph $G$ does not admit a three-dimensional super stable tensegrity realization if and only if $\lambda(G)\leq 2$.
By  Theorem~\ref{thm:equality}, the latter condition is equivalent to $\nu(G)\leq 3$. 
Van der Holst~\cite{H02} has shown that 
$\nu(G)\leq 3$ if and only if $G$ has no multigraph in the list of Figure~\ref{fig:list} as a minor. Thus, the statement follows.
\end{proof}

The two-dimensional counterpart of Theorem~\ref{thm:characterization_super_stable} also follows from the characterization of multigraphs with $\nu(G)\leq 2$ due to van der Holst~\cite{H02}. We will discuss a stronger statement in Section~\ref{sec:characterization_rd}.

In rigidity applications, it is more important to restrict our attention to injective realizations. 
Currently we have the following sufficient condition.
\begin{theorem}
Let $G$ be a $2$-connected multigraph.
Suppose $G$ contains $K_5, Q_3$ or $K_{2,2,2}$ as a minor.
Then $G$ has a three-dimensional injective super stable tensegrity realization.
\end{theorem}
\begin{proof}
For $K_5, Q_3, K_{2,2,2}$, there are well known three-dimensional super stable tensegrity realizations:
The realization of $K_5$ is as given in Example 3;
The realization of $Q_3$ is known as a dihedral start-shaped tensegrity and it is illustrated in Figure~\ref{fig:tensegrity}(b);
The realization of $K_{2,2,2}$ is known as a prism tensegrity and it is illustrated in Figure~\ref{fig:tensegrity}(a).
Those realizations are injective and satisfy the assumption in Lemma~\ref{lem:split deg}. So they admit non-splittable equilibrium stresses.
Hence, the statement follows from Theorem~\ref{thm:injective}.
\end{proof}
Establishing a complete characterization remains open. We pose it as an open problem.
\begin{problem}\label{prob:injective}
Characterize the class of multigraphs which can be realized as a three-dimensional  injective super stable tensegrity.
\end{problem}

It should be noted that  Alfakih~\cite{A17} gave a characterization of a graph which can be realized as a $d$-dimensional super stable tensegrity whose point configuration is in  general position. 

For generic tensegrities (where point configurations are generic), characterizations of the graphs of rigid or globally rigid tensegrities have been studied in~\cite{G21,JJK,JRS}.

\section{Relation Among Graph Parameters}\label{sec:relation}
Let $G$ be a multigraph.
A vertex subset $X$ in $G$ is called a {\em clique} if the simplified graph of the subgraph induced by $X$ is complete. 
A {\em maximal clique} is a clique which is not a proper subset of a  clique in $G$.
For a multigraph $G$, let $\omega(G)$ be the maximum size of a clique in $G$.
Also, let $\kappa(G)$ be the maximum vertex connectivity\footnote{In this paper, any $k$-connected graph $H$ satisfies $|V(H)|\geq k+1$ by definition.} over all minors of $G$,
and ${\rm tw}(G)$ be the treewidth of $G$. 
(See Section~\ref{sec:characterization_rd} for the definition.)

$\nu(G)$ and ${\rm rd}(G^=)$ have been studied extensively and relations to other combinatorial graph parameters are known. 
The following statement summarizes the current status.
\begin{corollary}\label{cor:basic_relation}
For a multigraph $G$, 
\[
\omega(G)-2 \leq \kappa(G)-1 \leq \nu(G)-1= \lambda(G)\leq {\rm rd}(G)\leq {\rm rd}(G^=)\leq {\rm tw}(G).
\]
\end{corollary}
\begin{proof}
By definition, $\omega(G)-1 \leq \kappa(G)$. 
Van der Holst~\cite{H02} has shown that $\kappa(G)\leq \nu(G)$. 
By Theorem~\ref{thm:equality}, $\lambda(G)=\nu(G)-1$.
$\lambda(G)\leq {\rm rd}(G)$ has been shown in Proposition~\ref{prop:rd_lambda}.
${\rm rd}(G)\leq {\rm rd}(G^=)$ follows from the definition.
${\rm rd}(G^=)\leq {\rm tw}(G)$ has been observed  by Belk and Connelly~\cite{BC}.
\end{proof}
The inequality between ${\rm rd}(G^=)$ and ${\rm tw}(G)$ can be strict as observed by Belk and Connelly~\cite{BC}. 
The inequality between ${\rm rd}(G)$ and  ${\rm rd}(G^=)$ can be strict since ${\rm rd}(K_n)=n-2$ and  ${\rm rd}(K_n^=)=n-1$.
It is also known that  $\nu(G)$ is lower bounded by some monotone function of ${\rm tw}(G)$, see, e.g.,~\cite{L}.
This follows from the fact that the triangular lattice $\Delta_r$ of width $r$ satisfies $\nu(\Delta_r)=r$
and any graph with large treewidth contains a triangular lattice of large width as a minor.
The triangular lattice also gives an example of graphs $G$ with small connectivity and 
$\nu(G)-1=\lambda(G)={\rm rd}(G)={\rm rd}(G^=)={\rm tw}(G)$.

Currently there is no example that separates $\lambda(G)$ and ${\rm rd}(G)$, and we conjecture that they are actually equal.

\begin{conjecture}\label{conj:rd}
For a multigraph $G$, $\lambda(G)={\rm rd}(G)$.
\end{conjecture}

The conjecture has an important implication in the context of rigidity
since the following property holds for multigraphs with $\lambda(G)={\rm rd}(G)$.
\begin{lemma}\label{lem:rigidity}
Suppose $d=\lambda(G)={\rm rd}(G)$ for a connected multigraph $G$.
Then $d$ is the maximum dimension in which $G$ admits a globally rigid/super stable/universally rigid tensegrity realization.
\end{lemma}
\begin{proof}
Denote by ${\rm gr}(G)$ and ${\rm ur}(G)$ the maximum dimensions in which $G$ admits a globally rigid and universally rigid tensegrity realization, respectively.
By definition, ${\rm gr}(G)\geq {\rm ur}(G)$ holds, and by Theorem~\ref{thm:connelly}, ${\rm ur}(G) \geq \lambda(G)$ holds.
We also have ${\rm rd}(G)\geq {\rm gr}(G)$ since any realization of $G$ in higher dimensional space than ${\rm rd}(G)$ has a deformation into a lower dimension space.
We thus obtain $d={\rm rd}(G)\geq {\rm gr}(G)\geq {\rm ur}(G)\geq \lambda(G)=d$, and the equality holds in each inequality.
\end{proof}


We have already pointed out  that  ${\rm rd}(G) < {\rm rd}(G^=)$ may hold   since ${\rm rd}(K_n)=n-2$ and  ${\rm rd}(K_n^=)=n-1$.
Currently we have no example for which the difference of those two parameters is more than one.  Hence we pose the following conjecture.
\begin{conjecture}\label{conj:difference_by_one}
For a multigraph $G$, ${\rm rd}(G^=)-1\leq {\rm rd}(G)\leq {\rm rd}(G^=)$.
\end{conjecture}

Belk and Connelly gave a characterization of simple graphs $G$ with ${\rm rd}(G^=)\leq k$ for $k\in \{1,2,3\}$.
In view of this, the following problem is natural for understanding  ${\rm rd}$.
\begin{problem}\label{prob:rd}
Characterize a class of multigraphs $G$ satisfying ${\rm rd}(G)={\rm rd}(G^=)$.
\end{problem}

\section{Characterizing Multigraphs with Bounded Realizable Dimension}\label{sec:characterization_rd}
In this section we confirm the conjectures in Section~\ref{sec:relation} for several special classes of multigraphs by computing the value of ${\rm rd}(G)$ explicitly.

Since $\lambda(G)\leq {\rm rd}(G)$,  giving a lower bound on ${\rm rd}(G)$ can be accomplished by constructing a super stable realization of $G$. The difficult direction is to give an upper bound of ${\rm rd}(G)$. Namely, we need to show how to fold each tensegrity realization of  $G$ down to a lower dimensional space.
(In this paper, {\em folding} is any procedure that finds a lower dimensional deformation of a given tensegrity.) 
Our general folding strategy  is to use a special type of tree decompositions, which we shall explain in the next subsection. We then show how to apply this strategy in several classes of multigraphs.

We first review basic terminology on tree decompositions.
A {\em tree decomposition} of a multigraph $G=(V,E)$ is a pair $(T, \cW)$, where $\cW$ is  a collection of vertex subsets of $V$ and $T$ is a tree on $\cW$ satisfying the follows three conditions:
(i) each vertex $v\in V$ belongs to at least one set in $\cW$,
(ii) for each edge $e\in E$, there is at least one set in $\cW$ that contains both endvertices of $e$, and
(iii)  for each $v\in V$, the subcollection $\cW_v=\{W \in \cW: v\in W\}$ induces a subtree $T_v$  in $T$.
Note that  a tree decomposition of a multigraph $G$ is a tree decomposition of ${\rm si}(G)$, and vice versa.
A subset $W\in \cW$ is called a {\em bag},
the {\em  width} of the tree decomposition is $\max\{|W|-1: W\in \cW\}$, and
the {\em treewidth} ${\rm tw}(G)$ of $G$ is the minimum width over all tree decompositions of $G$.
A tree decomposition whose width attains the minimum width is called {\em optimal}.

There always exists an optimal tree decomposition  $(T, \cW)$ such that $W_i\not \subseteq W_j$ for any distinct bags $W_i, W_j$. Hence, throughout the paper, we include this property in the definition of tree decompositions and assume that a tree decomposition always satisfies this property.

An important property of a tree decomposition $(T, {\cal W})$ is that 
for any edge $t_1t_2$ in $T$,
$W_{t_1}\cap W_{t_2}$ is a separator in $G$.
Moreover, denoting the components of $T-t_1t_2$ by $T_1$ and $T_2$,  $W_{t_1}\cap W_{t_2}$ separates 
$\bigcup_{t\in T_1} W_t$ and $\bigcup_{t\in T_2} W_t$.
See \cite[Lemma 12.3.1]{D}.

\subsection{Realizable Dimension and Tree Decompositions}
To demonstrate our idea, let us first look at  a multigraph $G$ with ${\rm si}(G)=K_n$.
For $G=K_n^=$, a bar-joint framework of the complete graph with $n$ vertices realized in general position in $\mathbb{R}^{d-1}$ does not admit an equivalent framework in lower-dimensional space. So, ${\rm rd}(K_n^=)=n-1$ holds.
(This can be also seen from ${\rm rd}(K_n^=)\geq\lambda(K_n^=)=n-1$).

We claim that ${\rm rd}(G)\leq n-2$ if $G\neq K_n^=$.
To see this, take a pair $u,v$  of vertices which are not linked by parallel edges.
Then, for any $(n-1)$-dimensional tensegrity realization $(G,\sigma,p)$ of $G$, one can fold it to a $(n-2)$-dimensional space by rotating $p(u)$ about a $(n-3)$-dimensional axis that contains the remaining $(n-2)$ points of  $(G,\sigma,p)$ in such a way that only the distance between $p(u)$ and $p(v)$ changes. (And the distance between $p(u)$ and $p(v)$ monotonically increases or decreases depending on the direction of rotation.)
This implies ${\rm rd}(G)\leq n-2$.
When ${\rm si}(G)=K_n$,  the equality also follows by ${\rm rd}(G)\geq \lambda(G)\geq \lambda(K_n)=n-2$.

This elementary geometric observation turns out to be a nontrivial tool if we combine it with tree decompositions.
Let $G$ be a multigraph and $(T, \cW)$ be a tree decomposition of $G$. 
We say that a bag $W$ is {\em lacking} if 
either $|W|-1$ is strictly smaller than the width of $(T, \cW)$ or 
there is a pair $u,v\in W$ such that $u$ and $v$ are not linked by parallel edges and $\{u,v\}$ is not contained in any other bag. 
We say that $(T, \cW)$  is {\em lacking} if every bag is lacking.

By Corollary~\ref{cor:basic_relation}, 
${\rm rd}(G)\leq {\rm tw}(G)$ for any multigraphs $G$.
The existence of a lacking optimal tree decomposition implies a better bound as follows.
\begin{lemma}\label{lem:lacking}
Let $G$ be a multigraph.
Suppose that an optimal tree decomposition of $G$ is lacking.
Then ${\rm rd}(G)\leq {\rm tw}(G)-1$.
\end{lemma}
\begin{proof}
Let $d={\rm tw}(G)-1$, and let $(G,\sigma,p)$ be a tensegrity realization of $G$ in dimension greater than $d$. Our goal is to fold the tensegrity into a $d$-dimensional space.

Let $(T,\cW)$ be an optimal lacking tree decomposition of $G$.
We say that a bag $W$ is {\em flat} if the affine span of $p(W)$ is $d$-dimensional. 
We first show how to fold $(G,\sigma,p)$ so that each bag becomes flat.
Suppose $W^*\in {\cal W}$ is a bag which is not flat. 
Since $(T,\cW)$ is lacking,  there is a pair $u,v\in W^*$ such that $u$ and $v$ are not linked by parallel edges and $\{u,v\}$ is not contained in any other bag.
Consider removing the node $W^*$ from $T$, and let $T_1,\dots, T_k$ be the connected components in the resulting forest.
Let $X_i$ be the set of vertices of $G$ which are contained in a bag in $T_i$, i.e., $X_i=\bigcup_{W\in T_i} W$, for each $i=1,\dots, k$.
Then, since $W^*$ is the only bag that contains $\{u,v\}$, the third property of the tree decomposition implies  
\begin{equation}\label{eq:tree_decomposition}
\{u,v\}\not\subset X_i
\end{equation}
for each $i=1.\dots, k$.

Let $A=(W^*-u)\cup \bigcup_{X_i: v\in X_i} X_i$ and $B=(W^*-v)\cup \bigcup_{X_i: v\notin X_i} X_i$.
By (\ref{eq:tree_decomposition}) and the third property of  tree decompositions, 
we have that 
$u\notin A$, $A\cap B=W^*-u-v$,
and  there is no edge between $A\setminus W^*$ and $B\setminus W^*$.
Let $G_A$ and $G_B$ be the subgraphs of $G$ induced by $A$ and $B$, respectively.
Then, $|V(G_A)\cap V(G_B)|=|W^*-u-v|=d$ and there is at most one edge between $V(G_A)\setminus V(G_B)$ and $V(G_A)\setminus V(G_B)$, which is a single edge between $u$ and $v$ if exists. 
Hence, one can rotate the sub-tensegrity of $G_A$ about a $(d-1)$-dimensional axis containing $p(V(G_A)\cap V(G_B))$ in such a way that $W^*$ becomes flat in the resulting tensegrity. 
Since each bag $W$ with $W\neq W^*$ is contained in either $G_A$ or $G_B$, by this rotation all points of $p(W)$ rotate simultaneously or remain stationary.
In particular, a flat bag remains flat.
Hence, this procedure does not affect the flatness of other bags, and we can apply the argument in each bag independently.

Thus we may suppose that each bag  is flat in $(G, \sigma, p)$.
Pick any bag $W_0$ in $T$ and regard $T$ as a rooted-tree whose root is  $W_0$.
Let $H$ be a $d$-dimensional subspace that contains $p(W_0)$,
and call a bag $W$ an $H$-bag if $p(W)\subseteq H$.
Let $T'$ be the maximum sub-rooted-tree (rooted at $W_0$) consisting of $H$-bags.
If $T=T'$, then all points of $(G,\sigma,p)$ are contained in $H$,
and the statement follows.
Hence, assume $T'\neq T$.
Pick a bag $W_i$ such that it is not contained in $T'$ and its parent $W_j$ is contained in $T'$.
By the definition of tree decompositions, $W_i\not \subseteq W_j$ and $W_j\not \subseteq W_i$. Hence we have $|W_i\cap W_j|\leq {\rm tw}(G)=d+1$. 
If the affine span of $p(W_i\cap W_j)$ is $d$-dimensional, then $p(W_i)$ is already contained in $H$
since $p(W_i)$ is flat and hence $p(W_i\cap W_j)$ affinely spans all points in $p(W_i)$.
This however contradicts the fact that $W_i$ is not contained in $T'$.
Hence the dimension of the affine span of $p(W_i\cap W_j)$ is less than $d$.
We consider the sub-tree $T_i$ of $T$ consisting of the descendants of  $W_i$  in $T$,
and consider rotating the sub-tensegrity induced by the union of the bags in $T_i$ about  the affine span of $p(W_i\cap W_j)$.
Since $W_i$ is flat and the affine dimension of $p(W_i\cap W_j)$ is less than $d$, 
we can perform this rotation such that all points of $W_i$ are contained in $H$.
Then the size of $T'$ becomes larger in the resulting tensegrity.
Applying this argument repeatedly, we obtain a tensegrity contained in $H$ as required.
\end{proof}
  
\subsection{Realizable Dimension of Chordal Graphs}
As an application, we shall give a characterization of  ${\rm rd}(G)$ for chordal graphs $G$.

A simple graph is said to be {\em chordal} if $G$ has no induced cycle of length more than three.
There are several equivalent characterizations of chordal (simple) graphs. 
We use a characterization in term of tree decompositions:
A graph $G$ is chordal if and only if 
$G$ has an optimal tree decomposition such that 
each bag is a clique (see, e.g.,~\cite[Proposition~12.3.6]{D}).

 We say that a multigraph $G$ is {\em chordal} if ${\rm si}(G)$ is chordal. Fallat and Mitchell~\cite{FM} showed that, for a chordal multigraph $G$, $\nu(G)\in \{{\rm tw}(G), {\rm tw}(G)+1\}$ holds, and moreover  $\nu(G)={\rm tw}(G)+1$ holds if and only if 
 there is a maximum clique
 in which each pair of vertices is either linked by parallel edges or is contained in other maximal clique. 
 The following statement complements to their result from the rigidity theory viewpoint. 
 \begin{theorem}\label{thm:chordal}
 Let $G$ be a chordal multigraph.
 Then $\lambda(G)=\nu(G)-1={\rm rd}(G) \in \{{\rm tw}(G)-1, {\rm tw}(G)\}$.
 Moreover, ${\rm rd}(G)={\rm tw}(G)$ holds if and only if  there is a maximum clique in which each pair of vertices is either linked by parallel edges or is contained in other maximal clique.
 \end{theorem}
 \begin{proof}
 By Theorem~\ref{thm:equality}, $\lambda(G)=\nu(G)-1$.
 By the theorem of Fallat and Mitchell~\cite{FM} on $\nu(G)$, $\lambda(G)=\nu(G)-1 \in \{{\rm tw}(G)-1, {\rm tw}(G)\}$, and $\lambda(G)={\rm tw}(G)$ holds if and only if  there is a maximum clique in which each pair of vertices is either linked by parallel edges or is contained in other maximal clique.
 It remains to prove $\lambda(G)={\rm rd}(G)$.
  
 Let $t={\rm tw}(G)+1$, and consider an optimal tree decomposition $(T,{\cal W})$ such that each bag is a clique. Since every clique is contained in some bag (see, e.g., \cite[Corollary 12.3.5]{D}),
 $W$ is a largest bag in ${\cal W}$ if and only if it forms a maximum clique in $G$.
 Hence,  $(T, {\cal W})$ is not lacking if and only if there is a maximum clique in which each pair of vertices is either linked by parallel edges or is contained in other maximal clique.
 So, $\lambda(G)={\rm tw}(G)$ if and only if $(T,{\cal W})$ is not lacking.

  Suppose $(T, {\cal W})$ is not lacking.
 Then  $\lambda(G)={\rm tw}(G)$.
Hence,  Corollary~\ref{cor:basic_relation} implies that  
${\rm tw}(G)= \lambda(G) \leq {\rm rd}(G)\leq {\rm tw}(G)$, and the equality holds in each inequality.

 Suppose $(T, {\cal W})$ is lacking. Then 
  ${\rm rd}(G)\leq {\rm tw}(G)-1$ by Lemma~\ref{lem:lacking}.
  Since a maximum clique of $G$ has size $t$,  
  the minor monotonicity of $\lambda$ implies
  ${\rm tw}(G)-1=t-2 = \lambda(K_t)\leq \lambda(G) \leq {\rm rd}(G)\leq  {\rm tw}(G)-1$,
 and the equality holds in each inequality.

 Thus, $\lambda(G)={\rm rd}(G)$ holds, and the proof is completed.
 \end{proof}

It is not difficult to show that ${\rm rd}(G^=)={\rm tw}(G)$ for chordal graphs $G$. 
This and Theorem~\ref{thm:chordal} gives a characterization of chordal multigraphs $G$ satisfying 
${\rm rd}(G)={\rm rd}(G^=)$.

\subsection{Realizable Dimension at Most One}
In  this section we give a characterization of graphs whose realizable dimension is at most one.
\begin{theorem}\label{thm:rd1}
The following are equivalent for a multigraph $G$.
\begin{description}
\item[(i)] ${\rm rd}(G)\leq 1$.
\item[(ii)] $\lambda(G)\leq 1$.
\item[(iii)] $G$ has no minor isomorphic to $K_4$ or $K_3^=$.
\end{description}
\end{theorem}
\begin{proof}
We have already shown ${\rm rd}(G)\geq \lambda(G)$ and $\lambda(K_4)=\lambda(K_3^=)= 2$.
So, by the minor monotonicity of ${\rm rd}(G)$, it suffices to show that (iii) implies (i).

Suppose $G$ has no minor isomorphic to $K_4$ or $K_3^=$.
We show ${\rm rd}(G)\leq 1$ by induction on $V(G)$.

We may assume that $G$ is 2-connected since otherwise
we can decompose $G$ into $G_1$ and $G_2$ such that $|V(G_1)\cap V(G_2)|=1$
and a one-dimensional deformation of any tensegrity realization of $G$ can be constructed from those of $G_1$ and $G_2$ obtained by induction.

A basic fact from graph theory says that, for a simple graph $H$, ${\rm tw}(H)\leq 2$ if and only if $H$ has no $K_4$ minor.
Since $G$ has no $K_4$ minor, $G$ admits a tree decomposition $(T,{\cal W})$ of width at most two.
We take an optimal tree decomposition $(T,{\cal W})$ such that the number of non-lacking bags is as small as possible.

If the tree decomposition is lacking, then we can apply Lemma~\ref{lem:lacking} to have ${\rm rd}(G)\leq {\rm tw}(G)-1=1$ and we are done.
So we may assume that there is at least one bag $W^*$ which is not lacking. 

Since $G$ is 2-connected and $W_i\not \subseteq W_j$ for distinct bags $W_i$ and $W_j$, the size of each bag is exactly  three.
Denote $W^*=\{v_1,v_2,v_3\}$,
and let $G_{1,2}$ be the subgraph of $G$ defined as the union of all paths between $v_1$ and $v_2$ internally disjoint from $v_3$.
Similarly, $G_{2,3}$ and $G_{3,1}$ are defined.
By the 2-connectivity of $G$ and the definition of $G_{i,j}$,
\begin{equation}\label{eq:last}
\text{$\{V(G_{1,2})\setminus \{v_1,v_2\}, V(G_{2,3})\setminus \{v_2,v_3\}, V(G_{3,1})\setminus \{v_1,v_3\}\}$ is a partition of $V(G)\setminus \{v_1,v_2,v_3\}$.}
\end{equation}

Suppose $G_{1,2}$ is a path between $v_1$ and $v_2$.
See Figure~\ref{fig:rd1} for an illustration.
We denote the sequence of the vertices of the path  by $v_1=u_1, u_2,\dots, u_m=v_2$.
By the definition of tree decompositions, $\{v_1,v_2\}$ separates $\{u_2,\dots, u_{m-1}\}$, 
and hence the path $G_{1,2}$ forms an ear in $G$.
Let ${\cal W}_0$ be the set of bags $W$ in $T$ satisfying $W\cap \{u_2,\dots, u_{m-1}\}\neq \emptyset$.
Since $P$ is an ear in $G$, we can obtain a new optimal tree decomposition $(T',{\cal W}')$ from $(T,{\cal W})$ by deleting ${\cal W}_0$ and then splitting $W^*$ into  the new path of bags $\{v_3,u_1,u_2\}, \{v_3,u_2,u_3\},\dots, \{v_3,u_{m-1},v_m\}$.
See Figure~\ref{fig:rd1}.
In the resulting tree decomposition, each new bag is lacking since each new bag consists of $v_3, u_i, u_{i+1}$ for some $i$ with $1\leq i\leq m-1$ and $G$ does not have parallel edges between $u_i$ and $u_{i+1}$. So, the number of non-lacking bags decreases at least by one.
This contradicts the choice of $(T,{\cal W})$.

\begin{figure}[t]
\centering
\begin{minipage}{0.45\textwidth}
\centering
\includegraphics[scale=0.7]{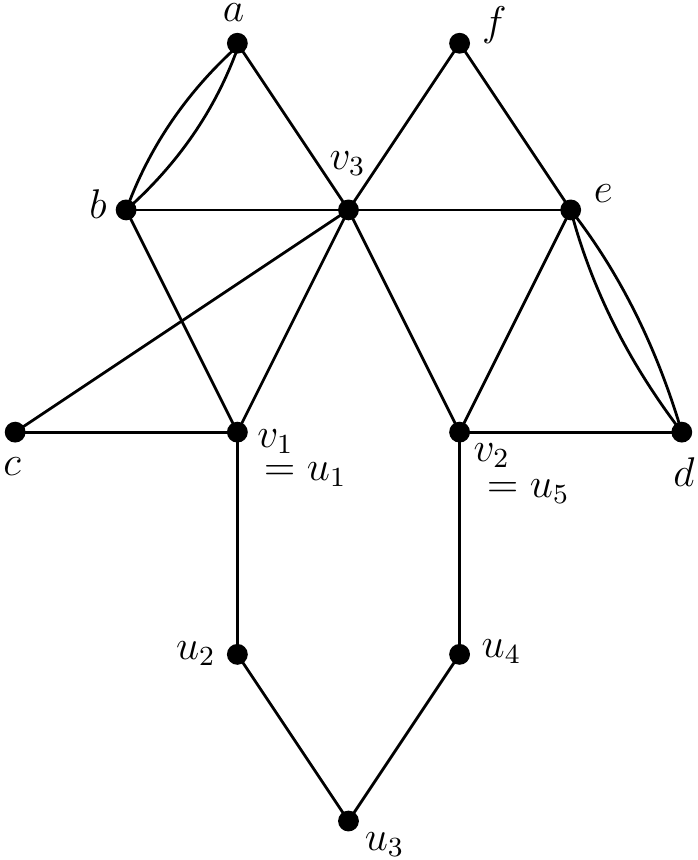}
\par
(a) $G$
\end{minipage}
\begin{minipage}{0.45\textwidth}
\centering
\includegraphics[scale=0.7]{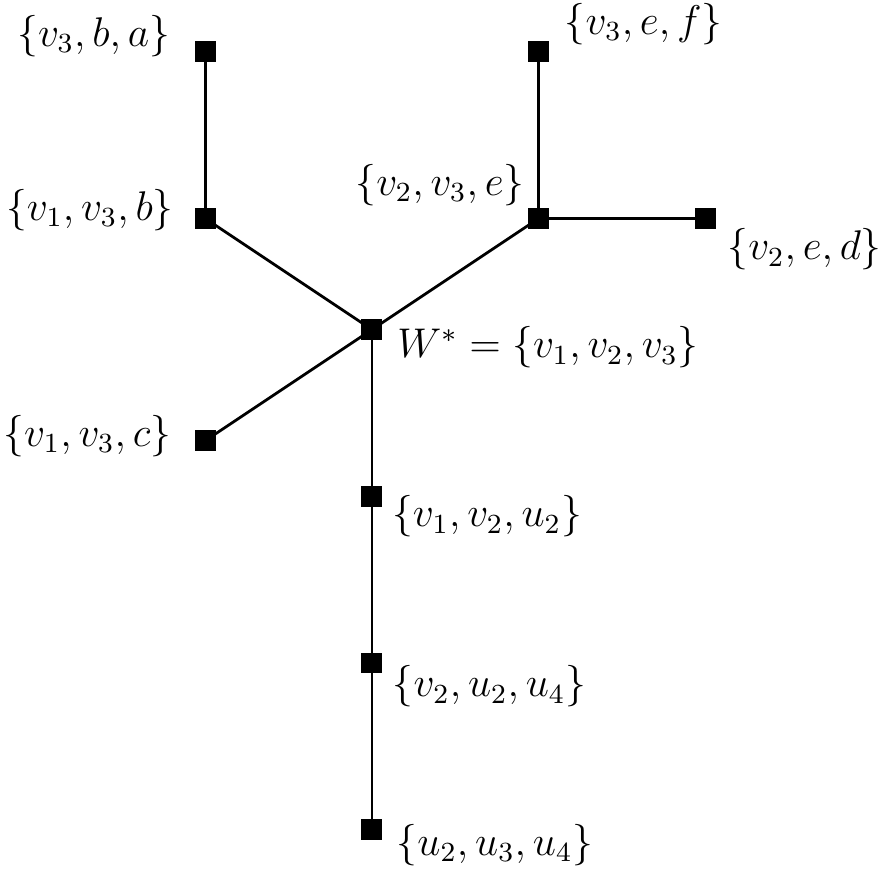}
\par
(b) $T$
\end{minipage}

\medskip

\begin{minipage}{0.45\textwidth}
\centering
\includegraphics[scale=0.7]{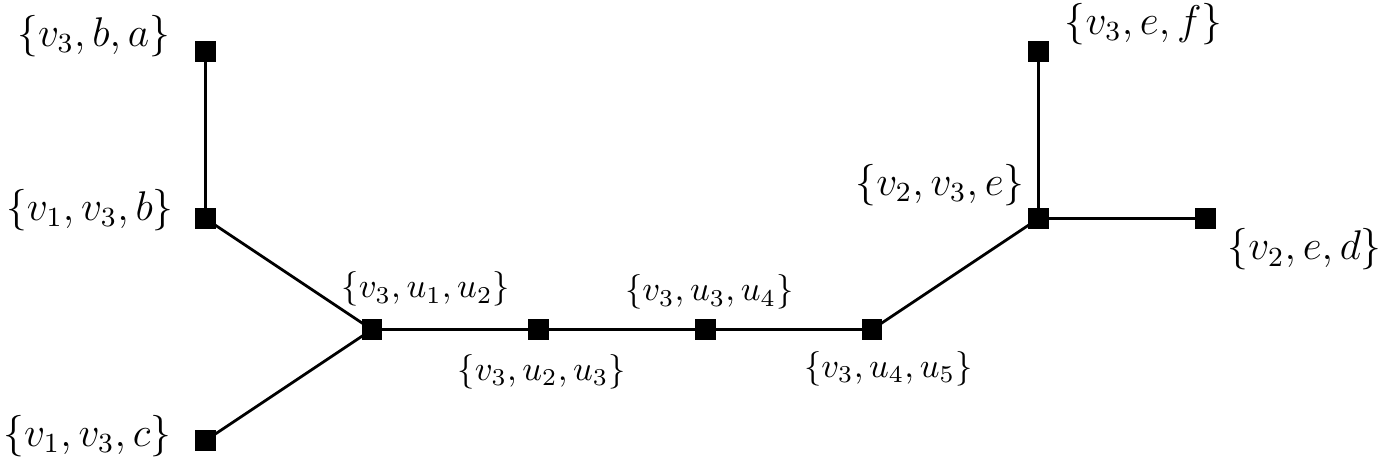}
\par
(c) $T'$
\end{minipage}
\caption{Proof of Theorem~\ref{thm:rd1}. For a given graph $G$ in Figure (a), Figure (b) is a tree decomposition $(T,{\cal W})$,
where the bag $W^*=\{v_1,v_2,v_3\}$ is not lacking.
$G_{1,2}$ is the path $u_1, u_2, u_3, u_4, u_5$ in Figure (a),
and ${\cal W}_0=\{\{v_1,v_2,u_2\}, \{v_2,u_2,u_4\}, \{u_2,u_3,u_4\}\}$ in Figure (b).
One can construct a new optimal tree decomposition as in Figure (c) without creating a non-lacking bag. 
}
\label{fig:rd1}
\end{figure}

Therefore, $G_{1,2}$ is not a path. Also $G_{1,2}$ is connected by the 2-connectivity of $G$. Symmetrically, $G_{i,j}$ is connected but not a path for any distinct $i,j$ with $i,j\in \{1,2,3\}$.
Then, in each $G_{i,j}$, there are at least two distinct paths between $v_i$ and $v_j$, which certify a $K_2^=$ minor rooted at $v_i$ and $v_j$.
The union of such a $K_2^=$ minor rooted at $v_i$ and $v_j$ over all distinct $i,j$ with $i,j\in \{1,2,3\}$ gives a $K_3^=$ minor in $G$ by (\ref{eq:last}).
This contradicts that $G$ has no $K_3^=$ minor.
\end{proof}

Characterizing multigraphs with realizable dimension at most two is an important open problem.
We formulate it as a conjecture.
\begin{conjecture}\label{con:rd2}
The following are equivalent for a multigraph $G$.
\begin{description}
\item[(i)] ${\rm rd}(G)\leq 2$.
\item[(ii)] $\lambda(G)\leq 2$.
\item[(iii)] $G$ has no minor isomorphic to a graph in Figure~\ref{fig:list}.
\end{description}
\end{conjecture}

The proof of Theorem~\ref{thm:rd1} shows that 
any multigraph having no $K_4$ or $K_3^=$ minor has a lacking optimal tree decomposition, and hence any realization can be folded to be one-dimensional by Lemma~\ref{lem:lacking}.
Such an argument is not enough for Conjecture~\ref{con:rd2}.
For example, the Wagner graph (see, e.g.,\cite{D}) with some parallel edges satisfies  the third  condition  in Conjecture~\ref{con:rd2}, but we do not know how to bound the realizable dimension  by two.

\paragraph{Acknowledgement} 
We thank the reviewer for their careful reading of the manuscript and their constructive comments.
The work was supported by 
JST PRESTO Grant Number JPMJPR2126
JST ERATO Grant Number JPMJER1903, and 
JSPS KAKENHI Grant Number 18K11155.

 \end{document}